\documentclass[reqno, 11pt]{amsart}
\usepackage{url}
\usepackage[colorlinks, linkcolor=blue, citecolor=blue, urlcolor=blue]{hyperref}
\usepackage{times}

\newtheorem{theorem}{Theorem}[section]
\newtheorem{lemma}[theorem]{Lemma}
\newtheorem{slemma}[theorem]{Sublemma}
\newtheorem{proposition}[theorem]{Proposition}

\theoremstyle{definition}
\newtheorem{definition}[theorem]{Definition}
\newtheorem{ex}[theorem]{Example}

\newtheorem{remark}[theorem]{Remark}

\numberwithin{equation}{section}

\usepackage{bbm}
\usepackage{url}
\usepackage{euscript}
\usepackage{pb-diagram}
\usepackage{lamsarrow}
\usepackage{pb-lams}
\usepackage{amsmath}
\usepackage{amsthm}

\usepackage{amsxtra}
\usepackage{amssymb}
\usepackage{pifont}
\usepackage{amsbsy}

\usepackage{graphicx}
\usepackage{epstopdf}

\newskip\aline \newskip\halfaline
\def\skipaline{\vskip\aline}

\def\qedbox{$\rlap{$\sqcap$}\sqcup$}
\def\qed{\nobreak\hfill\penalty250 \hbox{}\nobreak\hfill\qedbox\skipaline}

\def\proofend{\eqno{\mbox{\qedbox}}}

\newcommand{\one}{\mathbbm{1}}

\newcommand\bR{{\mathbb R}}

\newcommand\bZ{{\mathbb Z}}

\DeclareMathOperator{\tr}{{\rm tr}}

\DeclareMathOperator{\Gr}{\mathbf{Gr}}
 \DeclareMathOperator{\Hom}{Hom}
\DeclareMathOperator{\Vect}{Vect}

\DeclareMathOperator{\ev}{\mathbf{ev}}

\DeclareMathOperator{\Hess}{Hess}

\DeclareMathOperator{\var}{\boldsymbol{var}}
\DeclareMathOperator{\SO}{SO}
\DeclareMathOperator{\cov}{\boldsymbol{cov}}
\DeclareMathOperator{\Cov}{\boldsymbol{Cov}}

\newcommand{\be}{{\boldsymbol{e}}}
\newcommand{\bsf}{\boldsymbol{f}}

\newcommand{\ii}{\boldsymbol{i}}

\newcommand{\kk}{\boldsymbol{k}}
\newcommand{\bell}{\boldsymbol{\ell}}
\newcommand{\bm}{\boldsymbol{m}}

\newcommand{\bp}{{\boldsymbol{p}}}
\newcommand{\bq}{{\boldsymbol{q}}}

\newcommand{\bt}{\boldsymbol{t}}
\newcommand{\bu}{{\boldsymbol{u}}}
\newcommand{\bv}{{\boldsymbol{v}}}
\newcommand{\bw}{{\boldsymbol{w}}}
\newcommand{\bx}{{\boldsymbol{x}}}
\newcommand{\by}{{\boldsymbol{y}}}

\newcommand{\bsB}{\boldsymbol{B}}

\newcommand{\bsE}{\boldsymbol{E}}

\newcommand{\bsI}{\boldsymbol{I}}

\newcommand{\bsK}{{\boldsymbol{K}}}
\newcommand{\bsL}{\boldsymbol{L}}
\newcommand{\bsM}{\boldsymbol{M}}

\newcommand{\bsR}{\boldsymbol{R}}
\newcommand{\bsS}{\boldsymbol{S}}
\newcommand{\bsT}{\boldsymbol{T}}
\newcommand{\bsU}{{\boldsymbol{U}}}
\newcommand{\bsV}{\boldsymbol{V}}

\newcommand{\bsZ}{\boldsymbol{Z}}

\newcommand{\bgamma}{\boldsymbol{\gamma}}
\newcommand{\bGamma}{\boldsymbol{\Gamma}}

\newcommand{\blam}{{\boldsymbol{\lambda}}}

\newcommand{\brho}{{\boldsymbol{\rho}}}

\newcommand{\bom}{\boldsymbol{\omega}}
\newcommand{\bsi}{\boldsymbol{\sigma}}
\newcommand{\bSi}{{\boldsymbol{\Sigma}}}

\newcommand{\bpsi}{\boldsymbol{\psi}}

\newcommand{\bXi}{\boldsymbol{\Xi}}
\newcommand{\bOm}{\boldsymbol{\Omega}}

\newcommand{\si}{{\sigma}}

\newcommand{\ve}{{\varepsilon}}

\newcommand{\vfi}{{\varphi}}

\newcommand{\eA}{\EuScript{A}}

\newcommand{\eD}{\EuScript{D}}
\newcommand{\eE}{\EuScript{E}}

\newcommand{\eG}{\EuScript{G}}

\newcommand{\eI}{{\EuScript{I}}}

\newcommand{\eL}{\EuScript{L}}

\newcommand{\eN}{\EuScript{N}}
\newcommand{\eO}{\EuScript{O}}

\newcommand{\eS}{\EuScript{S}}

\newcommand{\eU}{\EuScript{U}}

\newcommand{\ra}{\rightarrow}

\newcommand{\Lra}{{\longrightarrow}}

\newcommand{\Llra}{{\Longleftrightarrow}}

\newcommand{\lan}{\langle}
\newcommand{\ran}{\rangle}

\def\inpr{\mathbin{\hbox to 6pt{\vrule height0.4pt width5pt depth0pt \kern-.4pt \vrule height6pt width0.4pt depth0pt\hss}}}

\newcommand{\pa}{\partial}

\newcommand{\dual}{\spcheck{}}

\begin{document}

\title[Random smooth functions]{Critical sets of random  smooth functions on compact manifolds} 

\date{Started  April 1, 2010. Completed  on January 26, 2011.
Last modified on {\today}. }

\author{Liviu I. Nicolaescu}
\thanks{This work was partially supported by the NSF grant, DMS-1005745.}

\address{Department of Mathematics, University of Notre Dame, Notre Dame, IN 46556-4618.}
\email{nicolaescu.1@nd.edu}
\urladdr{\url{http://www.nd.edu/~lnicolae/}}


\begin{abstract}  Given   a compact, connected Riemann manifold without boundary  $(M,g)$ of dimension $m$ and a  large positive constant $L$  we denote by $\bsU_L$ the subspace of $C^\infty(M)$ spanned by eigenfunctions  of the Laplacian corresponding  to eigenvalues $\leq L$. We equip $\bsU_L$ with the    standard Gaussian probability measure  induced by the $L^2$-metric on $\bsU_L$, and we denote by $\eN_L$ the  expected number  of    critical  points of a random function in $\bsU_L$. We prove  that $\eN_L\sim C_m\dim \bsU_L$ as $L\ra \infty$, where  $C_m$ is an explicit positive constant that depends only on the dimension $m$ of $M$ and satisfying the  asymptotic estimate $\log C_m\sim\frac{m}{2}\log m$ as $m\ra \infty$.\end{abstract}

\maketitle

\tableofcontents

\section{Introduction}
\setcounter{equation}{0}

Suppose that $(M,g)$ is a smooth, compact, connected  Riemann manifold of dimension $m>1$.   We denote by $|dV_g|$  the volume density     on $M$ induced by $g$. Throughout the paper we assume that the volume is normalized
\[
\int_M|dV_g(x)|=1.
\]
 For any $\bu, \bv\in C^\infty(M)$ we  denote by $(\bu,\bv)_g$ their $L^2$ inner product,
\[
(\bu,\bv)_g:=\int_M \bu(x)\bv(x) \,|dV_g(x)|.
\]
The $L^2$-norm  of a smooth function $u$ is then
\[
\|\bu\|:=\sqrt{(\bu,\bu)_g}.
\]
Let $\Delta_g: C^\infty(M)\ra C^\infty(M)$ denote the scalar Laplacian defined by the metric $g$.  For $L >0$ we set
\[
\bsU_L=\bsU_L(M,g):=\bigoplus_{\lambda\in [0,L]}\ker(\lambda-\Delta_g),\;\;d(L):=\dim \bsU_L.
\]
 We equip  $\bsU_L$  with the  Gaussian probability measure.
\[
d\bgamma_L(\bu):= (2\pi)^{-\frac{d(L)}{2}}e^{-\frac{\|\bu\|^2}{2}} |d\bu|.
\]
For  any $\bu\in\bsU_L$ we denote by   $\eN_L(\bu)$ the number of critical points of $\bu$.   If $L$ is sufficiently   large, then $\eN_L(\bu)$ is finite   with probability $1$.  We   obtain in this fashion a random variable  $\eN_L=\eN_{L,M,g}$, and we denote by $\bsE(\eN_L)$ its expectation
\[
\bsE(\eN_L):=\int_{\bsU_L} \eN_L(\bu) d\bgamma_L(\bu).
\]
In this paper   we   investigate the behavior of $\bsE(\eN_L)$ as $L\ra \infty$.   More precisely, we will prove the following result.
\begin{theorem}  For any $m>1$ there exists a positive constant $C=C(m)$  such that \textbf{\textit{for any}} compact, connected, $m$-dimensional  Riemannian manifold $M$ we have
\begin{equation}
\bsE(\eN_{L,M,g})\sim C(m) \dim \bsU_L(M,g)\;\;\mbox{as}\;\; L\ra \infty.
\label{eq: main1}
\end{equation}
\label{th: main}
\end{theorem}

The  constant $C(m)$ can be expressed in terms of certain statistics on the space  $\eS_m$ the space of symmetric $m\times m$ matrices .  We  denote  $d\bgamma_*$ the Gaussian measure\footnote{We refer to Appendix \ref{s: rand-sym} for a detailed   description of a   $3$-parameter family Gaussian measures $d\Gamma_{a,b,c}$ on $\eS_m$ that includes $d\bgamma_*$  as $d\bgamma_*=d\Gamma_{3,1,1}$.}on $\eS_m$ given by
\begin{equation}
\begin{split}
d\bgamma_*(X)&=\frac{1}{(2\pi)^{\frac{m(m+1)}{4}}\sqrt{\mu_m}}\,\cdot\,  e^{-\frac{1}{4}\bigl(\,\tr X^2-\frac{1}{m+2}(\tr X)^2\,\bigr)} 2^{\frac{1}{2}\binom{m}2}\prod_{i\leq j} dx_{ij},\\
\mu_m&=2^{\binom{m}{2}+1}(m+2)^{m-1}.
\end{split}
\label{eq: 311}
\end{equation}
Then
\begin{equation}
C(m)=\left(\frac{4\pi}{m+4}\right)^{\frac{m}{2}}\Gamma(1+\frac{m}{2})\, \underbrace{\int_{\eS_m}|\det X| \,d\bgamma_*(X)}_{=:I_m}.
\label{eq: main2}
\end{equation}
A similar result holds in the case  $m=1$.  In this case  $M=S^1$ and $\bsU_L$  is the space of    trigonometric   polynomials of degree $\leq L$. One can show (see \cite{N2})
\[
\bsE(\eN_{L,S^1})\sim  \sqrt{\frac{3}{5}}\dim\bsU_L\;\;\mbox{as}\;\;L\ra \infty.
\]
We can say something about the behavior of $C(m)$ as $m\ra \infty$.
\begin{theorem}\begin{equation}
\log C(m)\sim \log I_m \sim \frac{m}{2}\log m\;\;\mbox{as}\;\;m\ra \infty.
\label{eq: main3}
\end{equation}
\label{th: main2}
\end{theorem}
The proof of (\ref{eq: main1})  is based  on  Kac-Rice's   integral formula  \cite{AT, AzWs, BSZ2,DSZ1,DSZ2}   which expresses  the expected number of critical points of a function in $\bsU_L$ as an integral
\[
\bsE(\eN_L)=\int_M \rho_L(\bx)\,|dV_g(\bx)|.
\tag{$\ast$}
\label{tag: ast}
\]
The above  equality was given a  geometric interpretation by  Chern and Lashof \cite{CL}. More precisely,  they showed  that the    integral in the right-hand side of the above equality  is the  the total curvature of the immersion given by the evaluation map
\begin{equation}
  \ev  : M\ra \Hom(\bsU_L,\bR),\;\;\bp\mapsto \ev_\bp,
\label{eq: ev-intr}
\end{equation}
 where $\ev_\bp(\bu)=\bu(p)$, $\forall \bu\in\bsU_L$.  
 
 For our purposes  the probabilistic  description of the integrand $\rho_L(\bx)$  is more useful.  To formulate it let us denote by $\Hess_\bx(\bu,g)$ the Hessian  at $\bx$ of the random function $\bu\in \bsU_L$   computed using the  Levi-Civita connection  of the metric  $g$.  Using the metric  we   identify $\Hess_\bx(\bu,g)$ with a symmetric linear operator $T_\bx M\ra T_\bx M$.  Then  
 \begin{equation}
 \rho_L(\bx)=\frac{1}{\sqrt{\det 2\pi S_{d\bu(x)}}}\bsE\Bigl(\,|\det \Hess_\bx(\bu,g)|\; \bigr|\; d\bu(x)=0\,\Bigr).
 \label{eq: rho-prob}
 \end{equation}
 Above,  $S_{d\bu(x)}$   denotes  covariance matrix of  the Gaussian vector  $\bsU_L\ni \bu \mapsto d\bu(x)\in T_x^* M$, while  the quantity
 \[
 \bsE\Bigl(\,|\det \Hess_\bx(\bu,g)|\; \bigl|\; d\bu(x)=0\,\Bigr)
 \]
 is the conditional expectation  of the random variable  $\bu\mapsto |\det \Hess_\bx(\bu,g)|$ given that $d\bu(x)=0$.

Using the  \emph{regression formula} (see \cite[Prop. 1.2]{AzWs}  or (\ref{eq: regress1})) we express  this conditional expectation  as the     unconditional expectation of a new  random variable $|\det A_L(\bx)|$, where $A_L(\bx)$ denotes a  random, Gaussian  symmetric $m\times m$ matrix whose   covariance  takes into account the correlations   between the Gaussian variables $\bu\mapsto \Hess_\bx(\bu,g)$ and $\bu\mapsto d\bu(x)$.

 Next, we reduce  the large $L$ asymptotics of the  Gaussian random vector $d\bu(x)$ and matrix $A_L(\bx)$  to  questions concerning the asymptotics    of the spectral function $\eE_L$ of the Laplacian, i.e., the Schwartz kernel    of the  orthogonal projection onto $\bsU_L$.   These issues were addressed  in the pioneering work  of L. H\"{o}rmander \cite{Hspec}.     
   
  We  actually prove a bit more. We show that  
  \begin{equation}
  \lim_{L\ra \infty} L^{-\frac{m}{2}} \rho_L(\bx)= \frac{C(m)\bom_m}{(2\pi)^m},\;\;\mbox{uniformly in $\bx\in M$},
  \label{eq: main4}
  \end{equation}
  where $\bom_m$ denotes the volume of the unit ball in $\bR^m$.  Using the classical Weyl estimates (\ref{eq: weyl}) we see that (\ref{eq: main4}) implies  (\ref{eq: main1}).    
  
  The  equality (\ref{eq: main4}) has an interesting interpretation. We  can think of $\rho_L(\bx) |dV_g(\bx)|$  as the expected number of critical points of a random function in $\bsU_L$    inside an infinitesimal region  of volume $|dV_g(\bx)|$ around the point $\bx$. From this point of view  we see  that (\ref{eq: main4})  states that \emph{for large $L$ we expect the  critical points of   a random function   in $\bsU_L$  to be approximatively  uniformly distributed}.  
  
  We are inclined to believe  that  as $L\ra \infty$  the ratio 
  \[
  q_L=\frac{\var(\eN_L)}{\bsE(\eN_L)}
  \]
  has a finite limit $q(M,g)$.  Such a result would show that $\eN_L$ is highly concentrated near its mean value as $L\ra \infty$.   In  \cite{N2}  we proved   that this is the case when $M=S^1$ and moreover,  $q(S^1)\approx 0.4518....$.    In \cite{Ntorus} we proved that a closely related concentration result is valid for all  flat tori.  
  
  For a holomorphic counterpart of such an estimate we refer to \cite{SZ0}.

We obtain the  asymptotics  of $C(m)$    by relying on a trick   used by  Y.V. Fyodorov \cite{Fy} in a related context.      This reduces the asymptotics of the integral $I_m$ to known asymptotics of the $1$-point correlation function in random matrix theory, more precisely, Wigner's semi-circle law.

Philosophically, the  universality  result    contained in   Theorem \ref{th: main}  is a consequence   of  a universal  behavior of the spectral function   $\eE_L$   along the diagonal.  Roughly speaking, if we rescale the metric $g$ so that in the limit it becomes flatter, and flatter,  then the corresponding spectral function   begins to resemble the spectral   function of the Laplacian on the Euclidean   space $\bR^m$. For a precise formulation   of this universal rescaling phenomenon we refer to \cite{LPS, N3}.

  A related problem was considered by M. Douglas, B. Shiffman, S. Zelditch, \cite{DSZ1, DSZ2} where they investigate the  number of critical points of a random  holomorphic section of a  large power  $N$ of  a positive holomorphic line bundle $\eL$ over a K\"{a}hler manifold  $X$.      In these papers  the role of our $\bsU_L$ is played by the space of  holomorphic sections $H^0(X,\eL^N)$, and the large $L$ asymptotics is   replaced by large $N$ asymptotics.   The large $N$ asymptotics ultimately follow from the  refined asymptotics of the Szeg\"{o} kernels obtained by S. Zelditch in \cite{Z0}.      These refined  asymptotics   then lead to a complete asymptotic expansion as $N\ra\infty$ for the expected number of critical points of a random   holomorphic section of $\eL^N$.

  The proof of Theorem \ref{th: main} reveals several  additional interesting universal rescaling  phenomena.     We identify $\bsU_L$  with $\bsU_L\dual=\Hom(\bsU_L,\bR)$ using the $L^2$-metric. We can thus view the evaluation map in (\ref{eq: ev-intr}) as a map $\ev: M\ra \bsU_L$. For large  $L$ this map is an embedding, and we denote by $\bsi_L$ the pullback  to $M$  via $\ev$ of the $L^2$-metric on $\bsU_L$.  Equivalently, if $(\bpsi_k)$ is an orthonormal basis of $\bsU_L$, then
  \[
  \bsi_L=\sum_k d\bpsi_k\otimes d\bpsi_k.
  \]
   The  equality (\ref{eq: met-asy}) in the proof  of Theorem \ref{th: main} shows that the rescaled metric $g(L):=L^{-\frac{m+2}{2}}\bsi_L$ converges in the $C^0$ topology  to  $K_m g$, where $g$ is the original metric on $M$ and $K_m$ is a certain, explicit constant  that depends only on $m$; see (\ref{eq: km}).   This was also observed  by S. Zelditch,  \cite[Prop. 2.3]{Zel}.    A  closely  related result  was proved   in \cite[Thm.5]{BBG}.
  
     To obtain the convergence  of $g(L)$ in  stronger topologies we would need bounds    on the sectional curvature of $g(L)$. We     show that these bounds are equivalent to  some   refined asymptotic estimates    satisfied   by    certain linear combinations of fourth order  derivatives of the spectral function, (\ref{eq: bcurv}). 
     
  A  related embedding can be constructed  in the holomorphic case and S. Zelditch \cite{Z0} has proved that the resulting  sequence of   suitably rescaled metrics $g_N$  converges $C^\infty$ to the original K\"{a}hler metric.     The main reason  for such a stronger form of convergence is the  better behavior of the Szeg\"{o}  kernels.       Such a  regular behavior is not to be expected for the spectral function $\eE_L$.

The present paper is structured as follows. Section \ref{s: int}  contains the  formulation and the proof  of the key integral formula  (\ref{eq: rho-prob}), including  several reformulations   in the language of random  processes.   In this section we also present  a simple application of this formula to the number of critical points  of random spherical harmonics  of large degree on $S^2$. This sheds  additional  light on   a recent result of Nazarov and Sodin \cite{NS} on the number of nodal domains  of random spherical harmonics. More precisely,  the inequality (\ref{eq: zonal}) shows that  the expected number $\delta_n$  of zonal domain  on $S^2$ of a random harmonic polynomial of large degree $n$   satisfies the upper bound  $\delta_n <0.29 n^2$.

Section \ref{s: main1} contains the proof of the asymptotic  estimate (\ref{eq: main1}) and Section \ref{s: asy} contains the  proof of  the estimate (\ref{eq: main2}).

In our experience, many basic probabilistic technologies are not that familiar to an audience with a more geometric background. With this audience in mind   we decided to include  in Appendix \ref{s: gauss} a coordinate-free, brief survey of several facts about Gaussian measures and Gaussian processes   in a form adapted to the applications in this paper.  Appendix \ref{s: rand-sym} contains   a detailed  description of  a $3$-parameter  family of Gaussian measures on the space $\eS_m$ of real, symmetric $m\times m$ matrices. These measures  play a central role in the proof of (\ref{eq: main1})  and we could not find an appropriate  reference for the mostly elementary facts discussed in this  appendix. Appendix \ref{s: b} contains the computations of  a  Gaussian integral  involving random $2\times 2$ matrices.
\section*{Acknowledgments}

I  have greatly benefited  from discussions  with \fbox{Jianguo Cao},\footnote{He suddenly and untimely passed  away in June 2011. I will miss his generosity and  expertise.}  Richard Hind, Dima Jakobson and Steve   Zelditch  on  various parts of this  paper.

\section*{Notations}

\begin{enumerate}

\item For any random variable $\xi$ we denote by $\bsE(\xi)$ and respectively $\var(\xi)$ its expectation  and respectively its variance.

\item $\eS_m$ denotes the space of symmetric $m\times m$ real matrices.

\item For any finite dimensional real  vector space $\bsV$ we denote by $\bsV\dual$ its dual, $\bsV\dual:=\Hom(\bsV,\bR)$.

\item For any Euclidean space $\bsV$, we denote by $S(\bsV)$ the unit sphere in $\bsV$ centered at the origin and by $B(\bsV)$ the unit ball in $\bsV$ centered at the origin.

\item We will  denote by $\bsi_n$ the ``area'' of the  round  $n$-dimensional sphere $S^n$ of radius $1$, and by $\bom_n$  the  ``volume'' of the  unit ball in $\bR^n$.  These quantities are uniquely determined by the equalities (see \cite[Ex. 9.1.11]{N0})
\begin{equation*}
\bsi_{n-1}=n\bom_n=2\frac{\pi^{\frac{n}{2}}}{\Gamma(\frac{n}{2})}=\frac{n\pi^{\frac{n}{2}}}{\Gamma\bigl(1+\frac{n}{2}\bigr)},\;\;\Gamma\left(\frac{1}{2}\right)=\sqrt{\pi},
\tag{$\si$}
\label{tag: si}
\end{equation*}
where $\Gamma$  is Euler's Gamma function.

\item If $\bsV_0$ and $\bsV_1$ are two  Euclidean spaces of dimensions $n_0,n_1<\infty$ and $A:\bsV_0\ra \bsV_1$ is a linear map, then the \emph{Jacobian} of $A$ is the nonnegative scalar $J(A)$   defined as the norm of the linear map
\[
\Lambda^k A: \Lambda^k\bsV_0\ra \Lambda^k\bsV_1,\;\;k:=\min(n_0,n_1).
\]
More concretely, if  $n_0\leq n_1$,  and $\{\be_1,\dotsc,\be_{n_0}\}$ is an orthonormal basis of $\bsV_0$, then
\begin{equation*}
J(A)= \bigl(\,\det G(A)\,\bigr)^{1/2},
\tag{$J_-$}
\label{tag: j-}
\end{equation*}
where $G(A)$ is the $n_0\times n_0$ Gramm matrix with entries
\[
G_{ij}=\bigl(\, A\be_i, A\be_j\,\bigr)_{\bsV_1}.
\]
If $n_1\geq n_0$ then
\begin{equation*}
J(A)=J(A^\dag)=\bigl(\,\det G(A^\dag)\,\bigr)^{1/2},
\tag{$J_+$}
\label{tag: j+}
\end{equation*}
where $A^\dag$ denotes the adjoint (transpose) of $A$.   Equivalently,  if $d{\rm Vol}_i\in \Lambda^{n_i}\bsV_i^*$ denotes the  metric volume form  on $\bsV_1$, and $d{\rm Vol}_A$ denotes the metric volume  form on $\ker A$, then  $J(A)$ is the positive number such that
\begin{equation*}
d{\rm Vol}_0= \pm d{\rm Vol}_A\wedge  A^*d{\rm Vol}_1.
\tag{$J_+'$}
\label{tag: j+p}
\end{equation*}

\end{enumerate}

\section{A Kac-Rice type formula}
\label{s: int}
\setcounter{equation}{0}

 \subsection{The key integral formula} 
 \label{ss: CL}
 
 As  we mentioned in the introduction, a key component in the proof   of Theorem \ref{th: main} is an  integral formula that describes the  expected number of critical points as an integral over the background manifold $M$.  
 The literature     on random fields  contains many formul{\ae} of this type, and their proofs  follow  the strategy pioneered by M. Kac and S. Rice, \cite{AT, AzWs, Kac,Rice}.   
 
 We  believe that  it would greatly benefit  a  reader less fluent in the probabilistic language   to  first see  the   geometric origins  of these formulae.  For this reason  we decided  to include a complete  proof of  these formulae in our special case.  Not surprisingly the ubiquitous  double-fibration trick       in   integral geometry, \cite{APF, GGS, N0} will carry the day.  As a matter of fact, our  main integral formula (\ref{eq: KR}) contains  as  special cases the integral formul{\ae} of Chern-Lashof, \cite{CL} and  Milnor, \cite{Mil}.  
 
  Suppose  that $M$ is a smooth, compact, connected  manifold without boundary. Set $m:=\dim M$. 
 
 \begin{definition} (a) For any  nonnegative integer $k$, any point $\bp\in M$ and any $f\in C^\infty(M)$ we will denote by $j_k(f,\bp)$ the $k$-th jet of $f$ at $\bp$.

 \noindent (b)  Suppose that  $\bsU\subset  C^\infty(M)$ is a  linear subspace.  If $k$ is nonnegative integer then we say    that  $\bsU$ is $k$-\emph{ample} if for any $\bp\in M$ and any $f\in C^\infty(M)$ there exists $\bu\in\bsU$ such that
 \[
 j_k(\bu,\bp)=j_k(f,\bp).\proofend
 \]
\label{def: ample}
\end{definition}
 Fix   a finite dimensional vector space   $\bsU\subset  C^\infty(M)$ and set $N:=\dim\bsU$. We have an evaluation map
 \[
 \ev=\ev^{\bsU}: M\ra \bsU\dual:=\Hom(\bsU,\bR), \;\;\bp\mapsto\ev_\bp,
 \]
 where for any  $\bp\in M$ the linear map $\ev_\bp:\bsU\ra \bR$ is given by
 \[
 \ev_\bp(\bu)=\bu(\bp),\;\;\forall \bu\in \bsU.
 \]
For any $\bu\in C^\infty(M)$ we denote by  $\eN(\bu)$ the number of critical points  of $\bu$. In the remainder of this section we will assume  that $\bsU$ is $1$-ample.     This implies that   the evaluation map $\ev^{\bsU}$ is an immersion. Moreover, as explained in \cite[Cor. 1.26]{N1}, the $1$-ampleness  condition also implies  that  almost all functions $\bu\in \bsU$ are  Morse functions and thus    $\eN(\bu) <\infty$   for almost all $\bu\in\bsU$.

  We fix an inner product  $h=(-,-)_h$ on $\bsU$ and we denote by $|-|_h$ the resulting Euclidean norm. Using the metric $h$ we  identify $\bsU$ with its dual and thus we can regard the evaluation map as a smooth map $\ev:M\ra \bsU$.  We define the  expected number of critical points  of a function in $\bsU$ to be the  quantity
   \begin{equation}
   \eN(\bsU, h):= \frac{1}{\bsi_{N-1}}\int_{S(\bsU)} \eN(\bu)\, |dA_h(\bu)|=\int_{\bsU} \eN(\bu) \underbrace{\frac{ e^{-\frac{|\bu|_h^2}{2}}}{(2\pi)^{\frac{N}{2}}} \,|dV_h(u)|}_{=:d\bgamma_h(\bu)},
   \label{eq: expect}
   \end{equation}
   where $\bsi_{n-1}$  denotes the  ''area'' of the  unit sphere in $\bR^n$,  $|dA_h|$ denotes the  ''area'' density on the unit sphere  $S(\bsU)$, and $|dV_h(\bu)|$ denotes the volume density on $\bsU$  determined by the metric $h$.  A priori, the expected number of critical points could be infinite, but in any case, it is independent of any choice of  metric on $M$. The space $\bsU$ equipped  with the Gaussian probability measure $d\bgamma_h$ is a probability  space. We denote by $\eN_{U}$ the   random variable $\bsU\ni \nu\mapsto \eN(\bu)\in\bZ$ so that
  \[
  \eN(\bsU,h)=\bsE(\eN_{\bsU}, d\bgamma_h),
  \]
  where $\bsE(-,d\bgamma_h)$ denotes the expectation computed with respect to the probability measure   $d\bgamma_h$. We will refer to the pair $(\bsU, h)$ as the \emph{sample space}.

Fix a metric  $g$ on $M$.   We will express  $\eN(\bsU, h)$ as an integral
  \[
  \int_M \rho_g(\bp) \,|dV_g(\bp)|.
  \]
  The function $\rho_g$ does \emph{depend on $g$}, but the density  $\rho_g(\bp)\,|dV_g(\bp)|$ \emph{is independent of $g$}.   The  concrete description of $\rho_g(\bp)$    relies on  several  fundamental objects  naturally associated to the   triplet  $(\bsU,h, g)$.

  For any $\bp\in M$  we  set
  \[
  \bsU_\bp^0:=\bigl\{ \bu\in\bsU;\;\; d\bu(\bp)=0\bigr\}.
  \]
  The $1$-ampleness assumption on $\bsU$ implies that for any $\bp\in M$ the subspace $\bsU^0_\bp$ has codimension $m$ in $\bsU$ so that $\dim \bsU_\bp^0= N-m$. Denote by $dA_{S(\bsU_\bp^0)}$ the area density along the unit sphere $S(\bsU^0_\bp)\subset \bsU^0$.
 
 The differential  of the evaluation map at $\bp$  is a linear map $\eA_\bp: T_\bp M\ra \bsU$. We will refer to $\eA_\bp$ as the \emph{adjunction map} and we will denote by $J_g(\bp)=J_g(\bp, \bsU)$ its Jacobian. More precisely, if   $(\be_1,\dotsc,\be_m)$ is a $g$-orthonormal basis of $T_\bp M$,  then
  \[
  J_g(\bp)^2=\det \Bigl[\, \bigl(\,\eA_\bp \be_i, \eA_\bp \be_j\,\bigr)_h\,\Bigr]_{1\leq i,j\leq m}.
  \]
Since $\ev^\bsU$ is an immersion we have  $J_g(\bp)\neq 0$, $\forall \bx\in M$.

For any $\bp\in M$ and any  $\bu\in\bsU_\bp^0$, the Hessian of $\bu$ at $\bp$ is a well defined  symmetric bilinear form on $T_\bp M$ that can be identified  via the metric $g$ with a symmetric  endomorphism $\Hess_\bp(\bu,g)$ of $T_\bp M$. We denote this symmetric endomorphism by $\Hess_\bp(\bu, g)$.   

\begin{theorem}  If $(\bsU, h)$ is a $1$-ample   sample space on $M$, then
\begin{equation}
\begin{split}
\eN(\bsU, h) &=\frac{1}{\bsi_{N-1}}\int_M\frac{1}{J_g(\bp)}\left(\int_{S(\bsU_\bp^0)}  |\det\Hess_\bp(\bv,g)|\,|dA_{S(\bsU^0_\bp)}(\bv)|\right) |dV_g(\bp)|\\
&=(2\pi)^{-\frac{m}{2}}\int_M\frac{1}{J_g(\bp)}\,\underbrace{\left(\int_{\bsU_\bp^0} |\det\Hess_\bp(\bu,g)| \frac{e^{-\frac{|\bu|_h^2}{2}}}{(2\pi)^{\frac{N-m}{2}}} |dV_h(\bu)|\right)}_{=:I_\bp}\,|dV_g(\bp)|.
\end{split}
\label{eq: KR}
\end{equation}
\end{theorem}

\begin{proof}   Denote by $\underline{\bsU}_M$ the trivial vector bundle over $M$ with fiber $\bsU$, $\underline{\bsU}_M :=(\bsU\times M\ra M)$.   For any $\bp\in M$ we denote by $\bsK_\bp$ the orthogonal complement of $\bsU^0_\bp$ in  $\bsU$.
 
 \begin{lemma} The subspace  $\bsK_\bp$  coincides with the range of the adjunction map $\eA_\bp$.
 \end{lemma} 
 
 \begin{proof}   Indeed, if $(\Psi_n)_{1\leq n\leq N}$ is an orthonormal basis of $(\bsU, h)$, then
\[
\ev_\bp=\sum_n \Psi_n(\bp)\Psi_n\in \bsU.
\]
and for any  vector  field $X$ on $M$ we have
\[
\eA_\bp X_\bp= \sum_n (X\Psi_n)_{\bp}\Psi_n.
\]
Thus, the function $\bu=\sum_n u_n\Psi_n\in\bsU$, $u_n\in\bR$,  belongs  to $\bsK_\bp^\perp$ if and only if for any vector field $X$ on $M$ we have
\[
0=\sum_n u_n (X\Psi_n)_\bp= X\cdot \bu(\bp) \Llra \bu\in \bsU^0_\bp.
\]
\end{proof}
This proves that the collection $(\bsK_\bp)$ defines a subbundle $\bsK$ of $\underline{\bsU}_M$  and the adjunction map induces an isomorphism of vector bundle $\eA:TM\ra \bsK$.  We deduce that   the collection of spaces $(\bsU_\bp^0)_{\bp\in M}$  also forms a vector subbundle  $\bsU^0$ of the trivial bundle $\underline{\bsU}_M$ and we have an orthogonal  direct sum decomposition
\[ 
\underline{\bsU}_M=\bsU^0\oplus \bsK.
\]
For any section $u$ of $\underline{\bsU}_M$ we denote by $u^0$ its $\bsU^0$-component.

The  bundle $\underline{\bsU}_M$ is equipped with a canonical trivial  connection $D$. More precisely, if we regard a section of $u$ of  $\underline{\bsU}_M$ as a smooth map $u: M\ra \bsU$, then for any vector field $X$ on $M$    we  define $D_X u$ as the smooth function $M\ra \bsU$ obtained by derivating $u$ along $X$. The \emph{shape operator} of the  subbundle $\bsK$ is the bundle morphism $\bXi: TM\otimes \bsK\ra \bsU^0$ defined by the equality
\[
\bXi(X, \bu):= (D_X\bu)^0,\;\;\forall X\in C^\infty(TM),\;\;\bu\in C^\infty(\bsK).
\]
For every $\bp\in M$, we denote by $\bXi_\bp$ the induced linear map $\bXi_\bp: T_\bp M\otimes \bsK_\bp \ra \bsU^0$. If we denote by $\Gr_m(\bsU)$ the  Grassmannian of $m$-dimensional subspaces  of $\bsU$, then we have a  Gauss  map
\[
M\ni\bp \stackrel{\eG}{\longmapsto} \eG(\bp):=\bsK_\bp \in \Gr_m(\bsU).
\]
The  shape operator $\bXi_\bp$ can be viewed as a linear map
\[
\bXi_\bp : T_\bp M\ra \Hom(\bsK_\bp,\bsU^0_\bp)= T_{\bsK_\bp}\Gr_m(\bsU),
\]
  and, as such, it  can be identified with the differential of $\eG$ at $\bp$, \cite[\S 9.1.2]{N0}. Any  $\bv\in\bsU_\bp^0$ determines  a bilinear map
 \[
 \bXi_\bp\cdot\bv:  T_\bp M\otimes \bsK_\bp \ra\bR,\;\;\bXi_\bp\cdot\bv(\be, \bu):= \bigl(\,\bXi_\bp(\be,\bu), \bv\,\bigr)_h,
 \]
  By choosing   orthonormal bases  $(\be_i)$ in $T_\bp M$ and $(\bu_j)$ of $\bsK_\bp$ we can identify this bilinear form with an  $m\times m$-matrix.   This matrix depends on the choices of bases, but   the absolute value of its determinant is independent of these bases.   It is thus an invariant of the pair $(\bXi_\bp, \bv)$ that we will  denote by $|\det\bXi_\bp\cdot \bv|$.  
 
 \begin{lemma}
 \begin{equation}
\eN(\bsU, h)= \frac{1}{\bsi_{N-1}}\int_M{\left(\,\int_{S(\bsU^0_\bp)}|\det\bXi_\bp\cdot\bv|\, |dA_{S(\bsU^0_\bp)}(\bv)|\,\right)}\, |dV_g(\bp)|.
 \label{eq: av5}
\end{equation}
\end{lemma}
\begin{proof} Consider the incidence  variety
\[
\eI := \bigl\{ (\bp,\bv)\in M\times S(\bsU);\;\; d\bv(\bp)=0\,\bigr\}=\bigl\{ (\bx,\bv)\in M\times S(\bsU);\;\; \bv\in S(\bsU^0_\bp)\,\bigr\}.
\]
We have a natural double ``fibration''
\[
M\stackrel{\blam}{\longleftarrow} \eI\stackrel{\brho}{\longrightarrow} S(\bsU),
\]
where  the left/right projections $\blam,\brho$ are the canonical projections. The left projection $\blam:\eI\ra M$ describes $\eI$ as the unit sphere bundle associated to the  metric vector bundle $\bsU^0$. In particular, this shows that $\eI$ is a compact, smooth manifold of dimension $(N-1)$.  For generic $\bv\in S(\bsU)$ the  fiber $\rho^{-1}(\bv) $ is  finite  and can be  identified with  the set of critical points of $\bv:M\ra \bR$. We deduce
\begin{equation}
\eN(\bsU,h)=\frac{1}{{\rm area}\,(\,S(\bsU)\,)}\int_{S(\bsU)} \#\brho^{-1}(\bv)\, |dA_h(\bu)|.
\label{eq: av1}
\end{equation}
Denote by $g_\eI$ the metric on $\eI$ induced by the metric on $M\times S(\bsU)$ and by $|dV_\eI|$ the induced volume density.    The coarea   formula, \cite[\S 13.4]{BZ}, implies  that
\begin{equation}
\int_{S(\bsU)} \# \brho^{-1}(\bv) |dA_h(\bv)|=\int_\eI J_\brho(\bp,\bv) |dV_\eI(\bp,\bv)|,
\label{eq: av}
\end{equation}
where   the nonnegative  function $J_\brho$ is the Jacobian of $\brho$ defined by the equality
\[
\brho^*|dA_h|=J_\brho \cdot |dV_\eI|.
\]
To compute the integral in the right-hand side of (\ref{eq: av}) we need a  more explicit description of the geometry of the incidence variety $\eI$.

Fix a local orthonormal frame $(\be_1,\dotsc, \be_m)$ of $TM$  defined  in a neighborhood $\eO$ in $M$ of a given point $\bp_0\in M$. We denote by $(\be^1,\dotsc, \be^m)$ the dual co-frame of $T^*M$.  Set
\[
\bsf_i(\bp):=\eA_\bp\be_i(\bp)\in\bsU,\;\;i=1,\dotsc, m,\;\;\bp\in\eO.
\]
More explicitly, $\bsf_i(\bu)$ is defined by the equality
\begin{equation}
\bigl(\, \bsf_i(\bp),\bv\,\bigr)_h=\pa_{\be_i}\bu(\bp) ,\;\;\forall\bu\in\bsU.
\label{eq: fi}
\end{equation}
Fix a neighborhood $\eU\subset \lambda^{-1}(\eO)$ in $M\times S(\bsU)$  of the point  $(\bp_0,\bv_0)$,      and  a   local orthonormal frame $\bu_1 (\bp,\bv),\dotsc, \bu_{N-1}(\bp,\bv)$ over $\eU$ of the bundle $\rho^*TS(\bsU)\ra M\times S(\bsU)$ such that the following hold.

\begin{itemize}

\item The vectors $\bu_1(\bp,\bv),\dotsc,\bu_m(\bp,\bv)$ \emph{are independent  of the variable} $\bv$ and form an orthonormal basis of $K_\bx^\perp$. (E.g.,   we can obtain such vectors  from the vectors $\bsf_1(\bp),\dotsc, \bsf_m(\bp)$ via the  Gramm-Schmidt  process.)

\item  For $(\bp,\bv)\in \eU$, the space   $T_{\bp}E_{\bx}$ is spanned by  the vectors $\bu_{m+1}(\bp, \bv),\dotsc, \bu_{N-1}(\bp,\bv)$.

\end{itemize}

The   collection $\bu_1(\bp),\dotsc,\bu_m(\bp)$ is a collection of smooth sections of $\underline{\bsU}_M$ over $\eO$.  For any $\bp\in \eO$ and any $\be\in T_\bp M$,  we obtain    the vectors (functions).
\[
D_{\be}\bu_1(\bp),\dotsc, D_{\bp}\bu_m(\bx)\in \bsU,
\]
where  we recall that $D$ denotes the  trivial connection on $\underline{\bsU}_M$. Observe that
\begin{equation}
\eI\cap\eU=\bigl\{ (\bp,\bv)\in \eU;\;\; U_i(\bp,\bv)=0,\;\;\forall i=1,\dotsc, m\,\bigr\},
\label{eq: inc-eq}
\end{equation}
where $U_i$ is the function $U_i:\eO\times\bsU\ra \bR$ given by
\[
U_i(\bp,\bv):=\bigl(\, \bu_i(\bp),\bv\,\bigr)_{h}.
\]
Thus, the tangent  space  of  $\eI$ at $(\bp,\bv)$  consists of tangent vectors $\dot{\bp}\oplus \dot{\bv}\in  T_\bx M\oplus T_{\bv} S(\bsV)$ such that
\[
dU_i(\dot{\bp},\dot{\bv})=0,\;\;\forall i=1,\dotsc, m.
\]
We let $\omega_U$ denote the $m$-form
\[
\omega_U:= dU_1\wedge \cdots \wedge dU_m\in \Omega^m(\eU),
\]
and we denote  by  $\|\omega_U\|$   its norm  with respect to the product metric on $M\times S(\bsU)$.  Denote by $|\widehat{dV}|$ the volume density on $M\times S(\bsU)$ induced by the product metric. The equality (\ref{eq: inc-eq}) implies that
\[
|\widehat{dV}|=\frac{1}{\|\omega_U\|} \left|\omega_U\wedge dV_E\,\right|.
\]
Hence
\[
J_\brho|\widehat{dV}|= \frac{1}{\|\omega_U\|}|\omega_U\wedge \brho^* dA|.
\]
 We deduce
\[
J_\brho(\bp_0,\bv_0)=J_\brho(\bp_0,\bv_0)|\widehat{dV}|(\be_1,\dotsc,\be_m, \bu_1, \dotsc,\bu_{N-1})
\]
\[
= \frac{1}{\|\omega_U\|}| \omega_U\wedge \brho^* dS| (\be_1,\dotsc,\be_m, \bu_1, \dotsc,\bu_{N-1})=\frac{1}{\|\omega_U\|}\underbrace{\left| \omega_U\bigl(\, \be_1,\dotsc,\be_m\,\bigr)\right|_{(\bp_0,\bv_0)}}_{=:\Delta_U(\bp_0,\bv_0)}.
\]
Hence,
\begin{equation}
\int_{S(\bsU)} \# \brho^{-1}(\bw) |dA_h(\bv)|=\int_\eI \frac{\Delta_U}{\|\omega_U\|}\, |dV_\eI(\bp,\bv)|.
\label{eq: av3}
\end{equation}

\begin{slemma}  We have the equality 
\begin{equation}
J_\blam=\frac{1}{\|\omega_U\|},
\label{eq: prod-dens}
\end{equation}
where $J_\blam$ denotes the  Jacobian of the projection $\blam: \eI\ra M$.
\label{lemma: prod-dens}
\end{slemma}

\begin{proof}[Proof of Sublemma \ref{lemma: prod-dens}]   Along $\eU$  we have
\[
|\widehat{dV}|= \frac{1}{\|\omega_U\|} \left|\omega_U\wedge dV_\eI\,\right|,
\]
while  the definition of the Jacobian implies that 
\[
|dV_\eI| = \frac{1}{J_\blam}|dV_g\wedge dA_{S(\bsU^0_\bp)}|.
\]
 Therefore, it suffices to show that   along $\eU$ we have
\[
|\widehat{dV}|= |\omega_U\wedge dV_g\wedge dA_{S(\bsU^0_\bp)}|, 
\]
i.e.,
\[
\left|\omega_U\wedge  dV_g\wedge dA_{S(\bsU^0_\bp)}(\be_1,\dotsc, \be_m,\bu_1,\dotsc,\bu_{N-1})\,\right|=1.
\]
Since $dU_i(\bu_k)=0$, $\forall k\geq m+1$ we deduce  that
\[
\left|\omega_U\wedge  dV_g\wedge d A_{S(\bsU^0_\bp)}(\be_1,\dotsc, \be_m,\bu_1,\dotsc,\bu_{N-1})\,\right|=|\omega_U(\bu_1,\dotsc,\bu_m)|.
\]
Thus, it suffices to show that $|\omega_U(\bu_1,\dotsc,\bu_m)|=1$. This follows from the elementary identities
\[
dU_i(\bu_j)=(\bu_i,\bu_j)_{h}=\delta_{ij},\;\;\forall 1\leq i,j\leq m,
\]
where  $\delta_{ij}$ is the Kronecker  symbol.\end{proof}

Using  (\ref{eq: prod-dens}) in (\ref{eq: av3})  and the coarea formula we deduce
\begin{equation}
\int_{S(\bsU)} \# \brho^{-1}(\bw) |dA_h(\bv)|=\int_M\left(\,\int_{S(\bsU^0_\bp)}\Delta_U(\bp,\bv)\, |dA_{S(\bsU^0_\bp)}(\bv)|\,\right)\, |dV_g(\bp)|.
\label{eq: av4}
\end{equation}
Observe   that at a point $(\bp,\bv)\in \lambda^{-1}(\eO)\subset \eI$ we have
\[
dU_i(\be_j)=  \bigl(\, D_{\be_j}\bu_i(\bp),\bv\,\bigr)_{h}.
\]
We can rewrite this in terms of the shape operator $\bXi_\bp: T_\bp M\otimes \bsK_\bp\ra \bsU^0_\bp$.   More precisely,
\[
dU_i(\be_j)= \bigl(\, \bXi_\bp(\be_j,\bu_i),\bv\,\bigr)_{h}.
\]
Hence,
\[
\Delta_U(\bx, \bv)=\left| \det \bigl(\, \bXi_\bp(\be_j,\bu_i), \bv\,\bigr)_h\,\right|,
\]
We conclude that
\[
\int_{S_h(\bsU)} \# \brho^{-1}(\bv) |dA_h(\bv)|=\int_M{\left(\,\int_{S(\bsU_\bp^0)}|\det\bXi_\bp\cdot\bv|\, |dA_{S(\bsU^0_\bp)} (\bv)|\,\right)}\, |dV_M(\bp)|.
\]
This proves (\ref{eq: av5})
\end{proof}

To proceed further observe  that the left-hand side of  (\ref{eq: av5})  is   plainly independent of   the metric $g$ on $M$.     This raises the hope that if we judiciously  choose  the metric on $M$, then  we can  obtain a more manageable   expression for $\mu(M,\bsV)$. One choice presents itself.  

Let  $\bsi$ be   the pullback to $M$  of the metric on $\bsV$ via the immersion $\ev:M\ra \bsU$.  More concretely,  for any $\bp\in M$ and any $X,Y\in T_\bp M$, we have
\[
\bsi_\bp(X,Y)=\bigl(\,\eA_\bp X, \eA_\bp Y\,\bigr)_{h}.
\]
Fix  $\bp\in M$ and a $\bsi$-orthonormal  frame$ (\be_i)_{1\leq i\leq m}$ of $TM$ defined in a neighborhood $\eO$ of $\bp$.    Then  the collection $\bu_j=\eA \be_j$, $1\leq j$, is a local orthonormal frame of $\bsK|_{\eO}$.  The  shape operator has the simple description
\[
\bXi_\bp(\be_i,\bu_j)= \bigl(\,D_{\be_i}\eA\be_j\,\bigr)^0.
\]
Fix an orthonormal basis $(\Psi_n)_{1\leq n\leq N}$ of $\bsU$ so that every $\bv\in \bsU$ has a decomposition
\[
\bv=\sum_\alpha v_n \Psi_n,\;\;v_n\in\bR.
\]
Then
\[
\eA_\bp\be_j(\bp)=\sum_n (\pa_{\be_j} \Psi_n)_\bp\Psi_n,\;\; D_{\be_i}\eA^\dag\be_j(\bp)=\sum_n (\pa^2_{\be_i \be_j}\Psi_n)_\bp \Psi_n,
\]
and
\[
\bigl(\, (D_{\be_i}\eA\be_j)_\bp,\bv\,\bigr)_{h}= \sum_\alpha v_n (\pa^2_{\be_i \be_j}\Psi_\alpha)_\bp  = \pa^2_{\be_i\be_j}\bv(\bp).
\]
If $\bv\in \bsU^0_\bp$, then the Hessian of $\bv$ at $\bp$ is a well-defined,    symmetric bilinear form $\Hess_\bp(\bv)$ on $T_\bp M$ that can be identified via  the metric  $\bsi$ with a symmetric linear operator
\[
\Hess_\bp(\bv,\bsi) : T_\bp M\ra T_\bp M.
\]
If we fix a $\bsi$-orthonormal frame $(\be_i)$ of $T_\bp M$, then  the operator $\Hess_\bp(\bv,\bsi)$  is described by the symmetric $m\times m$ matrix with entries $\pa^2_{\be_i\be_j}\bv(\bx)$. We deduce that
\[
\left|\det\bXi_\bp \cdot \bv\right|=\left| \det \Hess_\bp(\bv,\bsi)\,\right|,\;\;\forall\bv\in S(\bsU^0_\bp).
\]
In particular, we deduce that
\begin{equation}
\eN(\bsU,h)=\frac{1}{\bsi_{N-1}}\int_M\,\left(\,\int_{S(\bsU^0_\bp)} |\det \Hess_\bp(\bv,\bsi)|\,|dA_{S(\bsU^0_\bp)}(\bv)\,\right)|dV_{\bsi}(\bp)|.
\label{eq: av6}
\end{equation}
This  is precisely the  main theorem of Chern and Lashof, \cite{CL}.

Finally, we  want to express (\ref{eq: av6}) entirely  in terms of  the  adjunction map $\eA$.   For any $\bp\in M$ and  any  $\bv\in \bsU_\bp$,  we define    the   density
\[
\mu_{\bp,\bv}:  \Lambda^m T_\bp M\ra \bR,
\]
\[
 \mu_{\bp,\bv}(X_1\wedge \cdots \wedge X_m)=    \left| \det\bigl( \, \pa^2_{X_iX_j}\bv (\bp)\,\bigr)_{1\leq i,j\leq m}\,\right|\,\cdot\,\bigl(\,\det \bigl(\,(\eA_\bp X_i,\eA_\bp X_j)_{h}\,\bigr)_{1\leq i,j\leq m}\,\bigr)^{-1/2}
\]
\[
=\left|\,\det\bigl(\,\Hess_\bp(\bv)(X_i,X_j)\,\bigr)_{1\leq i,j\leq m}\,\right| \cdot \bigl(\,\det\bigl(\, \bsi(X_i,X_j)\,\bigr)_{1\leq i,j\leq m}\,\bigr)^{-1/2},
\]
for any basis $X_1,\dotsc, X_m$ of $T_\bp M$. Observe that for any $\bsi$-orthonormal frame $\be_1,\dotsc, \be_m$ of $T_\bp M$ we have
\[
\mu_{\bp,\bv}(\be_1\wedge\cdots\wedge\be_m)= |\det \Hess_\bp(\bv, \bsi)\,|.
\]
If we integrate $\mu_{\bp,\bv}$ over $\bv\in S(\bsU^0_\bp)$, we obtain a density
\[
|d\mu_{\bsU}(\bp)|:  \Lambda^m T_\bp M\ra \bR,
\]
\[
 |d\mu_{\bsU}(\bp)|(X_1\wedge\cdots \wedge X_m)=\int_{S(\bsU_\bp)}  \mu_{\bp,\bv}(X_1\wedge \cdots \wedge X_m)\,|dA_{S(\bsU^0_\bp)}(\bv)|,
 \]
 $\forall X_1,\dotsc, X_m\in T_\bp M$.

Clearly  $|d\mu_{\bsU}(\bp)|$ varies smoothly with $\bp$, and thus it defines   a density $|d\mu_{\bsU}(-)|$ on $M$.  We want to emphasize that   this density \emph{depends on the metric  on} $\bsU$ but \emph{it is independent} of any metric on $M$. We will refer to it as \emph{the   density of $\bsU$}. By construction
\[
\eN(\bsU,h)=\frac{1}{\bsi_{n-1}}\int_M|d\mu_{\bsU}(\bp)|.
\]
If we  now return to our original  metric $g$ on $M$, then we can express    $|d\mu_{\bsU}(-)|$  as a product 
\[
|d\mu_{\bsU}(\bp)|=\delta_g(\bp)\cdot |dV_g(\bp)|,
\]
where $\delta_g=\delta_{g,\bsU}:M\ra \bR$ is a smooth nonnegative function. 

To find a  more useful description  of $\rho_g$, we choose  local coordinates $(x^1,\dotsc, x^m)$  near $\bp$ such that   $(\pa_{x^i})$ is a $g$-\emph{orthonormal} basis of $T_{\bp} M$. Then
\[
\mu_{\bp,\bv}(\pa_{x_1}\wedge \cdots\wedge\pa_{x_m}) =   \left| \det\bigl( \, \pa^2_{x_ix_j}\bv (\bp)\,\bigr)_{1\leq i,j\leq m}\,\right|\,\cdot\,\bigl(\,\det \bigl(\,(\eA_\bp\pa_{x_i},\eA_\bp \pa_{x_j})_{h}\,\bigr)_{1\leq i,j\leq m}\,\bigr)^{-1/2}.
\]
Observe  that the matrix $( \, \pa^2_{x_ix_j}\bv (\bp)\,\bigr)_{1\leq i,j\leq m}$   describes the Hessian operator
\[
\Hess_\bp(\bv,g): T_\bp M\ra T_\bp M
\]
induced by the Hessian of $\bv$ at $\bp$ and the metric $g$.

The scalar  $\bigl(\,\det \bigl(\,(\eA_\bp \pa_{x_i},\eA_\bp \pa_{x_j})_{h}\,\bigr)_{1\leq i,j\leq m}\,\bigr)^{1/2}$ is precisely the Jacobian $J_g(\bp)$ of the  adjunction map $\eA_\bp: T_\bp M\ra \bsU$ defined in terms of the metric $g$ on $T_\bp M$ and the metric $h$ on $\bsU$.  We set
\[
\Delta_\bx(\bsV,g):=  \int_{S(\bsU^0_\bp)} |\det \Hess_\bx(\bv, {{g}})|\,|dA_{S(\bsU^0_\bp)}(\bv)|.
\]
Since $|dV_g(\bp)|(\pa_{x_1}\wedge\cdots \wedge \pa_{x_m})=1$, we deduce
\begin{equation}
\delta_{g,\bsV}(\bp)= \Delta_\bp(\bsV,g)\cdot J_g(\bp)^{-1}.
\label{eq: rho-mu}
\end{equation}
This proves the first equality in (\ref{eq: KR}). The second equality follows from the first by invoking  (\ref{eq: int-ball}) and  the explicit formula (\ref{tag: si}) for $\bsi_{N-1}$.\end{proof}

\subsection{A Gaussian random field perspective}\label{ss: prob-int} For our concrete purposes it is convenient to give a probabilistic interpretation to the integral formula (\ref{eq: KR}).  For the reader's convenience we have gathered in Appendix \ref{s: gauss} the basic probabilistic notions and facts  needed  in the sequel.

  Consider again the metric  $\bsi=\bsi_\bsU$, the pullback of the metric  $h$ on $\bsU$ via the  evaluation map. We will refer to it as the  \emph{stochastic metric} associated to the sample space $(\bsU,h)$.  It is convenient to have a local description   of the stochastic metric.

Fix an orthonormal basis $\bpsi_1,\dotsc, \bpsi_N$ of  $\bsU$. The evaluation map $\ev^\bsU:M\ra \bsU$   is then  given by
\[
M\ni \bx \mapsto \sum_n \bpsi_n(\bx)\cdot\bpsi_n\in \bsU.
\]
If  $\bp\in M$ and $\eU$ is an open coordinate neighborhood   of $\bp$ with coordinates   $x=(x^1,\dotsc, x^m)$, then
\begin{equation}
\bsi_{\bp}(\pa_{x^i},  \pa_{x^j}) =\sum_n \frac{\pa\bpsi_n}{\pa x^i}(\bp)\frac{\pa\bpsi_n}{\pa x^j}(\bp),\;\;\forall 1\leq i, j \leq m.
\label{eq: sto-met}
\end{equation}
Note that if   the collection $(\pa_{x^i})_{1\leq i\leq m}$ forms a $g$-orthonormal frame of $T_\bp M$, then
\begin{equation}
J_g(\bp)^2= \det\Bigl[\, \bsi_{\bp}(\pa_{x^i},  \pa_{x^j})\Bigr]_{1\leq i,j\leq m}.
\label{eq: jac}
\end{equation}
To  the sample space  $(\bsU,h)$  we  associate in a tautological fashion a  Gaussian  random field on $M$ as follows.  The measure $d\bgamma_h$ in (\ref{eq: expect}) is a probability measure and thus $(\bsU, d\gamma_h)$ is naturally a probability space.  We have a  natural map
 \[
\xi: M\times \bsU\ra \bsR,\;\;M\times \bsU\ni (\bp,\bu)\mapsto \xi_\bp(\bu):=\bu(\bp).
 \]
 The collection of random variables  $(\xi_\bp)_{\bp\in M}$ is a Gaussian  random   field on $M$.  
 
 Using the orthonormal basis $(\bpsi_k)$ of $\bsU$ we  obtain  a linear  isometry
 \[
 \bR^N\ni \bt=(t_1,\dotsc, t_n)\mapsto \bu_{\bt}=\sum_k t_k\bpsi_k\in\bsU,
 \]
with inverse $\bu\mapsto t_k(\bu)=h(\bu,\bpsi_k)$. For any   $\bp\in M$ and any $\bt\in\bR^N$ we have
\[
\xi_\bp(\bu_{\bt})=   \sum_k t_k \bpsi_k(\bp).
\]
 The  covariance kernel  of this  field is the function $\eE=\eE_\bsU:M\times M\ra \bR$ given by
 \begin{equation}
 \begin{split}
 \eE(\bp,\bq)& =\bsE(\xi_\bp,\xi_\bq)=\sum_{j,k=1}^N \left(\int_{\bR^N} t_jt_k d\gamma_N(\bt)\right)\bpsi_j(\bp)\bpsi_k(\bq)\\
 &= \sum_{k=1}^M \bpsi_k(\bp)\bpsi_k(\bq),
 \end{split}
 \label{eq: kern}
 \end{equation}
where $d\gamma_N$ is the canonical  Gaussian measure on $\bR^N$.

 If $\bp\in M$  and $\eU$ is an open coordinate neighborhood of $\bp$ with coordinates $x=(x^1,\dotsc, x^m)$ such that $x(\bp)=0$, then   we can rewrite  (\ref{eq: sto-met}) in terms of the covariance kernel alone
 \begin{equation}
 \bsi_{\bp}(\pa_{x^i},  \pa_{x^j})=\frac{\pa^2 \eE(x,y)}{\pa x^i\pa y^j}|_{x=y=0}.
\label{eq: sto-met10}
\end{equation}
Note that any vector field $X$   determines a  new  Gaussian random field on $M$, the derivative of $\bu$ along $X$.  We  obtain the Gaussian random variables $\bu\mapsto (X\bu)_\bp$, $\bu\mapsto (Y\bu)_\bp$, and  we have
\begin{equation}
\bsi_{\bp}(X,Y) = \bsE\bigl(\, (X\bu)_\bp, (Y\bu)_\bp\,\bigr).
\label{eq:  sto-met1}
\end{equation}
The last equality justifies the attribute stochastic  attached to the metric  $\bsi$.

We denote by $\nabla$ the Levi-Civita    connection      of the metric $g$.    The Hessian   of a smooth function  $f:M\ra \bR$ with  respect to the metric $g$ is the symmetric $(0,2)$-tensor  $\nabla^2f$ on  $M$  defined by the equality
\begin{equation}
\nabla^2f(X,Y) := XYf-(\nabla_XY)f,\;\;\forall X,Y\in \Vect(M).
\label{eq: hess-civita}
\end{equation}
If $\bp$ is a critical point of $f$ then $\nabla^2_\bp f$ is the usual  Hessian    of $f$ at $\bp$.  More generally, if $(x^1,\dotsc, x^m)$ are $g$-\emph{normal} coordinates at $\bp$, then
\[
\nabla^2_\bp f(\pa_{x^i},\pa_{x^j}) =\pa^2_{x^ix^j}f(\bp),\;\;\forall 1\leq i,j\leq m.
\]
For any $\bp\in \bsM$ and any $f\in C^\infty(M)$  we use  the metric $g_\bp$  to identify the bilinear form  $\nabla^2_\bp f$ on $T_\bp M$ with an element of $\eS(T_\bp M)$, the vector space of  symmetric endomorphisms  of the Euclidean space $(T_\bp M, g_\bp)$. For any $\bp\in M$ we have two random Gaussian vectors
\[
\bsU\ni\bu\mapsto \nabla^2_\bp \bu\in \eS(T_\bp M),\;\; \bsU\ni\bu\mapsto d\bu(\bp)\in T^*_\bx M.
\]
Note that the expectation of both  random vectors  are trivial while (\ref{eq: sto-met10}) shows that  the covariance form of $d\bu(\bp)$  is the metric $\bsi_\bp$.

To   proceed further   we need to make an additional assumption  on  the sample space $\bsU$. Namely,  in the remainder of this section we will assume that it is  $2$-\emph{ample}.  In this case    the map
\[
\bsU\ni\bu\mapsto \nabla^2_\bp \bu\in \eS(T_\bp M)
\]
is surjective so  the Gaussian random vector $\nabla^2_\bp \bu$ is  nondegenerate.  A simple application of the co-area formula  shows  that the integral $I_\bp$ in (\ref{eq: KR}) can be expressed as a conditional expectation
\[
I_\bp=\bsE\bigl(\, |\det\nabla^2_\bp \bu|\,\bigl|\, d\bu(\bp)=0\,\bigr).
\]
Observing that
\begin{equation}
J_g(\bp)=(\det \bsS_{d\bu(p)})^{\frac{1}{2}},
\label{eq: jac-cov}
\end{equation}
we deduce that 
\begin{equation}
\eN(\bsU, h) =\frac{1}{(2\pi)^{\frac{m}{2}}}\int_M (\det \bsS_{d\bu(p)})^{-\frac{1}{2}} \bsE\bigl(\, |\det\nabla^2_\bp \bu|\,\bigl|\, d\bu(\bp)=0\,\bigr)\,|dV_g(\bp)|.
\label{eq: KR2.5}
\end{equation}
The last   equality  is   the main  conclusion of the  Expectation Metatheorem, \cite[Thm. 11.2.1]{AT} or the  expectation formula in \cite[Thm. 6.2]{AzWs}.    We can simplify the  equality (\ref{eq: KR2.5})  even more by taking full advantage of the Gaussian  nature of the various random  variables involved in this equality.

The covariance  form  of the pair of random variables $\nabla^2_\bp \bu$ and $d\bu(p)$ is   the bilinear map
\[
\Omega:\eS(T_\bp M)\dual \times T_\bp M\ra \bR,
\]
\[
 \Omega(\xi,\eta)= \bsE\bigl( \lan \xi , \nabla^2_\bp \bu\ran\cdot    \lan d\bu, \eta\ran\,\bigr),\;\;\forall \xi\in \eS_m\dual,\;\;\eta\in T_\bp M.
 \]
 Using the  natural inner products on $\eS(T_\bp M)$ and $T_\bp M$ defined by $g_\bp$ we can regard the  covariance form as a linear operator
 \[
 \bOm_\bp: T_\bp M \ra \eS(T_\bp M).
 \]
Similarly, we can identify the covariance forms  of $\nabla^2_\bp u$ and $du$ with symmetric positive definite  operators
\[
\bsS_{\nabla^2_\bp \bu} :\eS(T_\bp M)\ra \eS(T_\bp M)
\]
and respectively
\[
\bsS_{d\bu(\bp)} : T_\bp M\ra T_\bp M.
\]
 Using the  regression formula (\ref{eq: regress2}) we deduce that
  \begin{equation}
 \bsE\bigl(\, |\det\nabla^2_\bp \bu|\,\bigl|\, d\bu(\bp)=0\,\bigr)=\bsE(|\det Y_\bp|),
 \label{eq: regress4}
 \end{equation}
 where  $Y_\bp:\bsU\ra \eS(T_\bp M)$ is a Gaussian random vector with   mean value zero and covariance operator
 \begin{equation}
 \bXi_\bp=\bXi_{Y_\bp} := \bsS_{\nabla^2_\bp \bu} -\Omega \bsS_{d\bu(\bp)}^{-1}\Omega^\dag:\eS(T_\bp M)\ra \eS(T_\bp M).
 \label{eq: rand-matrix}
 \end{equation}
Since  $\bsU$ is $2$-ample the operator $\bXi_p$ is invertible  and we have
\begin{equation}
\bsE(|\det Y_\bp|)=(2\pi)^{-\frac{\dim \eS(T_\bp)}{2}}(\det \bXi_\bp)^{-\frac{1}{2}} \int_{\eS(T_\bp M)} |\det Y|\,e^{-\frac{(\bXi_\bp^{-1} Y,Y)}{2}}dV_g(Y).
\label{eq: cond-regr}
\end{equation}
 We deduce that \emph{when $\bsU$ is $2$-ample} we have
\begin{equation}
\eN(\bsU, h) =\frac{1}{(2\pi)^{\frac{m}{2}}}\int_M (\det \bsS_{d\bu(p)})^{-\frac{1}{2}} \bsE(|\det Y_\bp|)\,|dV_g(\bp)|,
\label{eq: KR3}
\end{equation}
where $Y_\bp$ is a   Gaussian random     symmetric endomorphism of $T_\bp M$ with expectation $0$ and covariance operator $\bXi_\bp$ described by (\ref{eq: rand-matrix}).

To  compute  the above integral we choose normal coordinates $(x^1,\dotsc, x^n)$ near $\bp$ and thus we can orthogonally identify $T_\bp M$ with $\bR^m$.  We can view the   Hessian   $\nabla^2_\bp \bu$ as a random variable
\[
H^\bp:\bsU\ra \eS_m:= \eS(\bR^m),\;\; \bsU\ni \bu \mapsto  H^\bp(\bu)\in \eS_m,\;\; H^\bp_{ij}(\bu)=\pa^2_{x^ix^j}\bu(\bp),
\]
and the differential $d\bu(\bp)$ as a random variable
\[
D^\bp: \bsU\ra \bR^m,\;\;\bu\mapsto D^\bp \bu \in \bR^m,\;\;D_i^\bp \bu=\pa_{x^i}\bu(\bp).
\]
The covariance operator $\bsS_{d\bu(p)}$ of the random variable $D^\bp$ is given by the symmetric $m\times m$ matrix with entries
\begin{equation}
\bsi_\bp(\pa_{x^i},\pa_{x^j})=\frac{\pa^2 \eE(x,y)}{\pa x^i\pa y^j}|_{x=y=0}.
\label{eq: bsi-loc}
\end{equation}
To compute the covariance form $\bSi_{H^\bp}$  of the  random matrix  $H^\bp$ we  observe first  that we have a canonical basis  $(\xi_{ij})_{1\leq i\leq j\leq m}$ of $\eS_m\dual$  so that $\xi_{ij}$ associates to a symmetric matrix $A$ the entry $a_{ij}$ located in the position $(i,j)$. Then
\begin{equation}
\begin{split}
\bSi_{H^\bp}(\xi_{ij},\xi_{k\ell}) &=\bsE\bigl(\, H^\bp_{ij}(\bu), H^\bp_{k\ell}(\bu)\,\bigr)= \bsE\bigl( \, \pa^2_{x^ix^j}\bu(\bx)\pa^2_{x^kx^\ell} u(\bx)\,\bigr)\\
&=\sum_{n=1}^N \pa^2_{x^ix^j}\bpsi_n(\bx)\pa^2_{x^kx^\ell}\bpsi_n(\bx)= \frac{\pa^4\eE(x,y)}{\pa x^i\pa x^j\pa y^k\pa y^\ell}|_{x=y=0}.
\end{split}
\label{eq: cov-hess}
\end{equation}
Similarly   we have
\begin{equation}
\Omega(\xi_{ij},\pa_{x^k})= \bsE\bigl(\, \pa^2_{x^i x^j}\bu(\bp), \pa_{x^k} \bu(\bp) \,\bigr) =\frac{\pa^3 \eE(x,y)}{\pa x^i\pa x^j \pa y^k}|_{x=y=0}.
\label{eq:  omega}
\end{equation}
To identify $\Omega$ with an operator  it suffices to observe that   $(\pa_{x^k})$ is an orthonormal basis of $T_\bp M$, while the collection  $\{\,\hat{\xi}_{ij}\,\}_{i\leq j}\subset \eS_m\dual$,
\[
\hat{\xi}_{ij}=\begin{cases}
\xi_{ij}, & i=j\\
\sqrt{2}\xi_{ij}, & i<j
\end{cases}
\]
is an orthonormal basis of $\eS_m\dual$. If we denote by $\widehat{E}_{ij}$ the dual  orthonormal basis of $\eS_m$, then
\[
\Omega \pa_{x^k}= \sum_{i\leq j} \Omega(\hat{\xi}_{ij}, \pa_{x^k})\widehat{E}_{ij}.
\]
\begin{remark} If the metric $g$ coincides with the stochastic metric $\bsi$, then  the covariance  operator  $\Omega$ is trivial. For a proof of this and of many other nice properties of the metric $\bsi$ we refer to \cite[\S 12.2]{AT}.\qed  
\label{rem: stoch}
\end{remark}

\subsection{Zonal domains of spherical harmonics of large degree}
\label{ss: zonal}
In the conclusion of this section we  want to discuss an immediate application  of the  above results to critical sets of random spherical harmonics.

Let $(M,g)$ be the  unit round  sphere $S^2$. The  spectrum of the Laplacian on $S^2$ is
\[
\lambda_n= n(n+1),\;\;n=0,1,2,\dotsc, \;\;\dim \ker(\lambda_n-\Delta)= 2n+1=d_n.
\]
The space $\bsU_n=\ker(\lambda_n-\Delta)$ has a well known descrition: it consists of sperical harmonics, i.e., restrictions to $S^2$ of  harmonic   polynomials  of degree $n$ in three variables. We want to   describe the behavior of $\eN(\bsU_n)$ as $n\ra \infty$, where $\bsU_n$ is equipped    with the $L^2$-metric. In other words we     want to find the expected   number of  critical points of a spherical harmonic of very large degree.

In this case   the covariance kernel  $\eE_n(\bp,\bq)$  of $\bsU_n$ has a  very simple description. More precisely, if $(\Psi_k)_{1\leq k\leq 2n+1}$ is an orthonormal basis of $\bsU_n$, then the classical addition theorem, \cite[\S 1.2]{Mu} shows that
\[
\eE_n(\bp,\bq)=\sum_k \Psi_k(\bp)\Psi_k(\bq)=\frac{2n+1}{4\pi} P_n(\bp\bullet\bq),\;\;\forall\bp,\bq\in S^2,
\] 
where $\bullet$ denotes the inner product in $\bR^3$, and $P_n$ denotes the $n$-th Legendre polynomial,
\[
P_n(t)=(-1)^n\frac{1}{2^nn!}\frac{d^n}{dt^n}(1-t^2)^n.
\]
In this case the stochastic metric  $\bsi=\bsi_n$ is obviously $SO(3)$-invariant  and  it is a  (constant) multiple  of the  round metric.  In view of  Remark \ref{rem: stoch}  this implies that for any $\bp\in S^2$ the random variables
\[
\bsU_n\ni \bu\mapsto \Hess_\bp(\bu, g)\;\;\mbox{and}\;\;\bsU_n\ni \bu\mapsto d\bu(\bp)
\]
are independent  and we deduce that
\[
\eN(\bsU_n)=\frac{1}{2\pi}\int_{S^2} \frac{1}{J_g(p)}\left(\int_{\bsU_n} |\det \Hess_\bp(\bu, g))| \underbrace{\frac{e^{-\frac{1}{2}|\bu|^2}}{(2\pi)^{\frac{\dim \bsU_n}{2}}}|d\bu|}_{=:d\bgamma_n(\bu)}\right) dV_g(\bp).
\]
Clearly, the  integrand in the above formula is invariant with respect to the $SO(3)$-action on $S^2$ and we thus have
\begin{equation}
\eN(\bsU_n)=\frac{2}{J_g(\bp_0)}\int_{\bsU_n}  \bigl|\det \Hess_{\bp_0}(\bu,g) \bigr|\, d\gamma_n(\bu),
\label{eq: harm10}
\end{equation}
where $\bp_0$ a  fixed (but arbitrary) point on $S^2$.  To compute the term in the right-hand side of the above equality we use the equalities  (\ref{eq: bsi-loc}) and (\ref{eq: cov-hess}). 

Fix  normal coordinates $(x^1,x^2)$ in a neighborhood $\eO$ of   $\bp_0$ so we can view $\eE_n$ as a function $\eE_n(x,y)$. The location of a point $\bp\in\eO$ is described  by a smooth   function 
\[
\eO\ni (x^1,x^2)\mapsto\bp(x^1,x^2)\in\bR^3. 
\]
The tangent vector $\pa_{x^i}$,  viewed as a  vector in $\bR^3$, corresponds with the derivative $\bp_{x^i}:=\pa_{x^i}\bp$ of the above function. At $\bp_0$ we have
\begin{equation}
\bp_{x^i}\bullet\bp_{x^j}=\delta_{ij}\;\;\mbox{and}\;\;\bp_{x^i}\bullet\bp_0=0,\;\;\forall i,j.
\label{eq: norm-second}
\end{equation}
The  arcs  $C_1=\{x^2=0\}$ and $C_2=\{x^1=0\}$ are portions of great circles intersecting  orthogonally at $\bp_0$.  Note that $x^1$ is  the arclength parameter along $C_i$, $i=1,2$.  The vectors  $\bp_{x^i}$ are unit  tangent vectors along these arcs.    This shows that  at $\bp_0$ we have
\[
\bp_{x^ix^i}=-\bp_0.
\]
Since the arcs  $C_1$ and $C_2$ are planar their torsion is trivial   and the Frenet formul{\ae} imply  that at $\bp_0$ we have
\[
\bp_{x^ix^j}=0,\;\;\forall i\neq j.
\]  
The last two equalities can be rewritten in compact form as
\begin{equation}
\bp_{x^ix^j}=-\delta_{ij}\bp_0,\;\;\forall i,j
\label{eq: norm-second2}
\end{equation}
We set
\begin{equation}
\begin{split}
s_n&:=\frac{2n+1}{\pi} P_n'(1)=\frac{2n+1}{4\pi}\times \frac{n(n+1)}{2}\sim \frac{1}{4\pi}n^3\\
t_n& :=\frac{2n+1}{\pi} P_n''(1)=\frac{(2n+1)}{4\pi}\times\frac{(n+2)(n+1)n(n-1)}{16}\sim\frac{1}{32\pi}n^5.
\end{split}
\label{eq: st}
\end{equation}
We deduce
\begin{equation}
\begin{split}
\bsi(\pa_{x^j},\pa_{x^k}) & =\pa_{x_j}\pa_{y^k}\eE(\bp,\bq)|_{\bp=\bq=\bp_0}\\
& = \frac{2n+1}{4\pi}\left(\, P_n'(\bp\bullet\bq) \bp_{x^j}\bullet\bq_{y_k} +P_n^{(2)}(\bp\bullet\bq)(\bp_{x^j}\bullet\bq)(\bp\bullet\bq_{y^k}) \,\right)_{\bp=\bq=\bp_0}\\
&=s_n\delta_{jk},
\end{split}
\label{eq: bsi-harm}
\end{equation}
and
\begin{equation}
J_g(\bp_0)=s_n.
\label{eq: harm30}
\end{equation}

To compute  $\pa^4_{x^ix^jy^ky^\ell}\eE_n(\bp,\bq)$ at $\bp=\bq=\bp_0$ we will use (\ref{eq: norm-second})  and  (\ref{eq: norm-second2})  to cut down the complexity of the final formula. We deduce that  at $\bp=\bq=\bp_0$ we have
\[
\pa^4_{x^ix^jy^ky^\ell}\eE_n(\bp,\bq)  = \frac{2n+1}{4\pi}\left(\, P_n'(\bp\bullet\bq) \bp_{x^ix^j}\bullet\bq_{y^\ell y^k} +P_n^{(2)}(\bp\bullet\bq)(\bp_{x^ix^j}\bullet\bq)(\bp\bullet\bq_{y^ky^\ell}) \,\right)_{\bp=\bq}
\]
\[
+\frac{2n+1}{4\pi}\left(\, P_n^{(2)}(\bp\bullet\bq) (\bp_{x^i}\bullet\bq_{y^\ell})(\bp_{x^j}\bullet\bq_{y^k}) +P_n^{(2)}(\bp\bullet\bq)(\bp_{x^j}\bullet\bq_{y^\ell})(\bp_{x^i}\bullet\bq_{y^k}) \,\right)_{\bp=\bq},
\]
and  thus
\begin{equation}
\pa^4_{x^ix^jy^ky^\ell}\eE_n(\bp,\bq)_{\bp=\bq}= \bigl(s_n+t_n\bigr)\delta_{ij}\delta_{k\ell} +t_n\bigl(\,\delta_{i\ell}\delta_{jk}+\delta_{ik}\delta_{j\ell}\,\bigr).
\label{eq: harm20}
\end{equation}
Denote by $d\bGamma_n$ the pushforward  of the Gaussian measure  $d\bgamma_n$ via the Hessian map
\[
\bsU_n\ni \bu\mapsto \Hess_{\bp_0}(\bu,g)\in \eS(T_{\bp_0} S^2)=\eS_2.
\]
We deduce  from (\ref{eq: harm20}) that the covariance form $\bSi_n$ of $d\bGamma_n$ satisfies the equality
\[
\bSi_n=\bSi_{a_n,b_n, c_n},\;\; a_n= s_n+3t_n,\;\;b_n=s_n+t_n,\;\;c_n=t_n,
\]
where $\bSi_{a,b,c}$ is defined by the conditions (\ref{eq: cov-diag}) and (\ref{eq: cov-off}). Observe that  $a_n,b_n,c_n$ satisfy (\ref{eq: inv}),i.e., $a_n=b_n+2c_n$. As explained in Appendix \ref{s: rand-sym},   this implies that   $d\bGamma_n$ is $O(2)$-invariant.  Set
\[
a^*_n=\frac{a_n}{t_n},\;\;b_n^*=\frac{b_n}{t_n},\;\; c_n^*=\frac{c_n}{t_n},
\]
and denote by $d\bGamma_n^*$  the Gaussian  measure on $\eS_2$ with covariance matrix $\bSi_{a_n^*,b_n^*,c_n^*}$. Using (\ref{eq: resc-gauss}) we deduce that
\[
\int_{\eS_2}|\det X|\, d\bGamma_n(X)=t_n  \int_{\eS_2}|\det X|\, d\bGamma^*_n(X).
\]
From (\ref{eq: harm10}) and (\ref{eq: harm30}) we now deduce
\begin{equation}
\eN(\bsU_n)= \frac{2t_n}{s_n} \int_{\eS_2}|\det X|\, d\bGamma^*_n(X).
\label{eq: harm40}
\end{equation}
Observe that  as $n\ra \infty$ we have
\[
\frac{2t_n}{s_n}\sim \frac{n^2}{4},\;\;a_n^*\sim 3,\;\;b_n^*\sim 1 \;\; c_n^*\sim 1,
\]
so that 
\begin{equation}
\eN(\bsU_n)\sim\frac{n^2}{4} \int_{\eS_2}|\det X|\, d\bGamma_{3,1,1}(X),
\label{eq: harm50}
\end{equation}
where $d\bGamma_{3,1,1}(X)$ is the Gaussian  measure on $\eS_2$ with covariance form $\bSi_{3,1,1}$. More precisely (see (\ref{eq: cov-gauss1}))
\[
d\bGamma_{3,1,1}(X)=\frac{1}{4 (2\pi)^{ \frac{3}{2} } }  e^{-\frac{1}{4}\bigl(\,\tr X^2 -\frac{1}{4}(\tr X)^2\,\bigr)}\,\cdot\, \sqrt{2}\prod_{1\leq i\leq j\leq 2} dx_{ij}.
\]
In Appendix \ref{s: b} we show that
\begin{equation}
\int_{\eS_2}|\det X|\, d\bGamma_{3,1,1}(X)=\frac{4}{\sqrt{3}},
\label{eq: harm60}
\end{equation}
and we deduce from (\ref{eq: harm10})  that
\begin{equation}
\eN(\bsU_n)\sim\frac{n^2}{\sqrt{3}}\;\;\mbox{as}\;\; n\ra \infty.
\label{eq: harm70}
\end{equation}

Let us observe that for $n$ very large, a typical spherical harmonic $\bu\in\bsU_n$ is a Morse function on $S^2$ and $0$ is a regular value. The  nodal set $\{\bu=0\}$ is  disjoint union of smoothly embedded  circles.  We denote by $\eD_\bu$ the set of connected components of the complement  of the nodal set are called the  \emph{nodal domains} of $\bu$ and  we denote $\delta(\bu)$ the cardinality of $\eD_\bu$.  A result of  Pleijel and Peetre, \cite{BeMe, Pee, Plei},  shows  that
 \begin{equation}
 \delta(\bu)\leq \frac{4}{j_0^2} n^2\approx 0.692 n^2,
 \label{eq: pleij}
 \end{equation}
 where $j_0$ denotes the first positive  zero of the Bessel function $J_0$. 
 
 We  think of $\delta(\bu)$ as a random variable and we denote by $\delta_n$ its expectation,
 \[
 \delta_n=\frac{1}{(2\pi)^{\dim \bsU_n}{2}}\int_{\bsU_n} \delta(\bu) e^{-\frac{1}{2}|\bu|^2} |d\bu|.
 \]

Denote by $p(\bu)$ the number of local minima and maxima of $\bu$, and by $s(\bu)$ the number of saddle points. Then
\[
\eN(\by)=p(\bu)+s(\bu),\;\;p(\bu)-s(\bu)=\chi(S^2)=2.
\]
This proves that
\[
p(\by)=\frac{1}{2}\bigl(\,\eN(\bu)+2\,\bigr).
\]
For every    nodal region  $D$,  we denote by $p(\bu, D)$ the number of local minima and maxima\footnote{A simple application of the  maximum principle shows that on  each  nodal domain, all the local extrema of $\by$ are of the same type: either all local minima or all local maxima.  Thus $p(\bu, D)$ can be visualized as the number of  \textit{\textbf{p}}eaks of $|\bu|$ on $D$.} of $\bu$ on $D$.   Note that $p(\bu, D)>0$ for any $D$ and thus the number  $p(\bu)=\sum_{D\in \eD_\bu} p(\bu, D)$ can be viewed as a weighted  count of      nodal domains.  Moreover
\[
\delta(\bu)\leq p(\bu).
\]
We set
\[
p(\bsU_n):=\frac{1}{(2\pi)^{\frac{\dim\bsU_n}{2}}}\int_{\bsU_n} e^{-\frac{1}{2}|\bu|^2}p(\bu)\,|d\bu|.
\]
The equality (\ref{eq: harm70})  implies that
\[
p(\bsU_n)\sim \frac{1}{2\sqrt{3}}n^2\;\;\mbox{as}\;\;n\ra \infty,\;\;  \frac{1}{2\sqrt{3}} \approx 0.288.
\]
This  shows that while 
\[
\max_{\bu\in\bsU_n}\delta (\bu)\leq 0.692 n^2,
\]
the  expectation  $\delta_n$ is less than half this theoretical maximum,
\[
\delta_n\approx 0.288
\]
Recently, Nazarov and Sodin \cite{NS},  have proved  that there exists a positive  constant   $a>0$ such that
\[
\delta_n\sim an^2  \;\;\mbox{as}\;\;n\ra \infty.
\]
Additionally, for large $n$,      with high probability,  $\delta(\bu)$ is close to  $an^2$ (see \cite{NS} for a precise statement).   This shows that 
\begin{equation}
a\leq \frac{1}{2\sqrt{3}} \approx 0.288. 
\label{eq: zonal}
\end{equation}
More information about  lower bounds on $a$ can be found in  Maria N\u{a}st\u{a}sescu's senior thesis \cite{Nast}.

\section{The proof of Theorem \ref{th: main}}
\label{s: main1}
\setcounter{equation}{0}

\subsection{Asymptotic estimates of the spectral function} We fix an orthonormal basis  of $L^2(M,g)$ consisting of eigenfunctions  $\Psi_n$ of $\Delta_g$,
\[
\Delta_g\Psi_n=\lambda_n\Psi_n,\;\;n=0,1,\dotsc,\;\;\lambda_0\leq \lambda_1\leq \cdots \leq \lambda_n\leq \cdots.
\]
The collection $(\Psi_n)_{\lambda_n\leq L}$ is therefore an orthonormal  basis of    $\bsU_L$ so that the covariance kernel of the  Gaussian field determined by $\bsU_L$ is
\[
\eE_L(\bp,\bq)=\sum_{\lambda_n\leq L}\Psi_n(\bp)\Psi_n(\bq).
\]
This function is also known as  the \emph{spectral function} associated to the   Laplacian.   Equivalently, $\eE_L$ can be identified with the Schwartz kernel of the orthogonal projection onto $\bsU_L$.   Observe that
\[
\int_M \eE_L(\bp,\bp)\,|dV_g(\bp)|=\dim\bsU_L .
\]
In  the groundbreaking work \cite{Hspec},  L. H\"{o}rmander used  the kernel  of the wave group $e^{\ii t \sqrt{\Delta}}$  to produce refined asymptotic estimates for  the spectral  function.  More precisely he showed (see \cite{Hspec} or \cite[\S 17.5]{H3})
\begin{equation}
\eE_L(p,p) = \frac{\bom_m}{(2\pi)^m}L^{\frac{m}{2}}+O\bigl( L^{\frac{m-1}{2}}\,\bigr)\;\;\mbox{as}\;\;L\ra \infty,
\label{eq: hor}
\end{equation}
uniformly with respect to $\bp\in M$. Above, $\bom_m$ denotes the volume of the unit ball in $\bR^m$. This implies immediately the classical Weyl estimates
\begin{equation}
\dim \bsU_L\sim \frac{\bom_m}{(2\pi)^m}{\rm vol}_g(M)L^{\frac{m}{2}}.
\label{eq: weyl}
\end{equation}
H\"{o}rmander's approach    can be  refined  to produce   asymptotic estimates    for  the behavior of  the derivatives the spectral function  in a  neighborhood of the diagonal. We describe below  these estimates following closely the presentation in \cite{Bin}. For more general results we refer to \cite[Thm. 1.8.5, 1.8.7]{SV}.

We set $\lambda:= L^{\frac{1}{2}}$.  Fix a point $\bp$ and    normal coordinates   $x=(x^1,\dotsc, x^m)$ at $\bp$.    Note that $x(\bp)=0$.   For any  multi-indices $\alpha,\beta\in \bZ^m_{\geq 0}$ we have (see \cite[Thm. 1.1, Prop. 2.3]{Bin})
\begin{equation}
 \frac{\pa^{\alpha+\beta}\eE_L(x,y)}{\pa x^\alpha\pa y^\beta}|_{x=y=0}= C_m(\alpha,\beta)\lambda^{m+|\alpha|+|\beta|} + O\bigl(\, \lambda^{m+|\alpha|+|\beta|-1}\,\bigr),
\label{eq: bin}
\end{equation}
where
\begin{equation}
C_m(\alpha,\beta)=\begin{cases}
0, & \alpha-\beta\not\in(2\bZ)^m,\\
&\\
\frac{(-1)^{\frac{|\alpha|-|\beta|}{2}}}{(2\pi)^m}\int_{\bsB^m} \bx^{\alpha+\beta} |d\bx|, & \alpha-\beta\in (2\bZ)^m,
\end{cases}
\label{eq: bin-const}
\end{equation}
and $\bsB^m$ denotes the unit ball
\[
\bsB^m=\bigl\{ \bx\in \bR^m;\;\;|\bx|=1\,\bigr\}.
\]
The estimates (\ref{eq: bin}) are uniform in $\bp\in M$.    Using (\ref{eq: int-ball}) we deduce (compare with (\ref{eq: inv1}) )
\[
\frac{1}{(2\pi)^m}  \int_{\bsB^m} \bx^{\alpha+\beta} |d\bx|=\frac{1}{(4\pi)^{\frac{m}{2}}\Gamma\bigl(1+\frac{|\alpha|+|\beta|+m}{2}\bigr)}\int_{\bR^m} \bx^{\alpha+\beta}\frac{e^{-|\bx|^2}}{\pi^{\frac{m}{2}}} |d\bx|.
\]
We set
\[
K_{m}=C_m(\alpha,\alpha),\;\;|\alpha|=1,
\]
so that
\begin{equation}
K_{m}= \frac{1}{(4\pi)^{\frac{m}{2}}\Gamma\bigl(2+\frac{m}{2}\bigr)}\int_{\bR^m} x_1^2\frac{e^{-|\bx|^2}}{\pi^{\frac{m}{2}}} |d\bx|=\frac{1}{2(4\pi)^{\frac{m}{2}}\Gamma\bigl(2+\frac{m}{2}\bigr)}.
\label{eq: km}
\end{equation}
For  any $i\leq j$  define $\alpha_{ij}\in \bZ^m$ so that
\[
\bx^{\alpha_{ij}}=x_ix_j.
\]
For $i\leq j$ and $k\leq \ell$ we set
\begin{equation}
C_m(i,j;k,\ell) = C_m(\alpha_{ij}, \alpha_{k\ell})=\frac{1}{(4\pi)^{\frac{m}{2}}\Gamma\bigl(3+\frac{m}{2}\bigr)}\int_{\bR^m} x_ix_jx_kx_\ell\frac{e^{-|\bx|^2}}{\pi^{\frac{m}{2}}} |d\bx|.
\label{eq: ijkl}
\end{equation}
For $i<j$ we have
\begin{equation}
C_m(i,i; j,j)=\frac{1}{(4\pi)^{\frac{m}{2}}\Gamma\bigl(3+\frac{m}{2}\bigr)}\int_{\bR^m} x_i^2x_j^2\frac{e^{-|\bx|^2}}{\pi^{\frac{m}{2}}} |d\bx| = \frac{1}{4(4\pi)^{\frac{m}{2}}\Gamma\bigl(3+\frac{m}{2}\bigr)}=: c_m.
\label{eq: ijij}
\end{equation}
\[
C_m(i,j; i,j) =C_m(i,i;j,j),
\]
Finally
\begin{equation}
C_m(i,i;i,i)=\frac{1}{(4\pi)^{\frac{m}{2}}\Gamma\bigl(3+\frac{m}{2}\bigr)}\int_{\bR^m} x_i^4\frac{e^{-|\bx|^2}}{\pi^{\frac{m}{2}}} |d\bx|= \frac{3}{4(4\pi)^{\frac{m}{2}}\Gamma\bigl(3+\frac{m}{2}\bigr)}=3c_m,
\label{eq: iiii}
\end{equation}
and
\[
C_m(i,j;k,\ell)=0,\;\;\forall k\leq \ell,\;\;(i,j)\neq (k,\ell).
\]

\subsection{Probabilistic consequences of the previous estimates}We denote by $\bsi^L$ the stochastic metric on $M$ determiner by the sample space $\bsU_L$, $L\gg 0$.    As explained in Subsection \ref{ss: prob-int} the covariance form  of the    random vector $\bsU_L\ni \bu \mapsto d\bu(\bp)\in T^*_\bp M$ is  $\bsi_\bp^L$, and  from (\ref{eq: bin}) we deduce
\begin{equation}
\begin{split}
\bsi_\bp^L(\pa_{x^i},\pa_{x^j}) & =\frac{\pa^2 \eE_L(x,y)}{\pa x^i\pa y^j}|_{x=y=0}= K_m\lambda^{m+2}\delta_{ij} +O(\lambda^{m+1})\\
&=K_m\lambda^{m+2} g_\bp(\pa_{x^i},\pa_{x^j})+O(\lambda^{m+1})\;\;\mbox{as}\;\; L\ra \infty,\;\;\mbox{uniformly in $\bp$. }
\end{split}
\label{eq: met-asy}
\end{equation}
In particular, if $\bsS^L_{d\bu(\bp)}$ denotes the   covariance operator of the random  vector $d\bu(\bp)$, then we deduce from the above equality that
\begin{equation}
\bsS^L_{d\bu(\bp)}= K_m\lambda^{m+2}\one_m+O(\lambda^{m+1}),\;\;\mbox{uniformly in $\bp$},
\label{eq: met-asy1}
\end{equation}
and invoking (\ref{eq: jac-cov}) we deduce
\begin{equation}
J^L_g(\bp)=(\det \bsS^L_{d\bu(\bp)})^{\frac{1}{2}}= K_m^{\frac{m}{2}}\lambda^{\frac{m(m+2)}{2}}+ O\bigl(\lambda^{\frac{m(m+2)}{2}-1}\bigr),\;\;\mbox{uniformly in $\bp$. }
\label{eq: jac-asy}
\end{equation}
Denote by $\bSi^L_{H^{\bp}}$ the covariance form of the random      matrix
\[
\bsU_L\ni\bu\mapsto \nabla^2_\bp \bu\in \eS(T_\bp M)=\eS_m.
\]
Using  (\ref{eq: cov-hess}) and (\ref{eq: bin})   we deduce
\begin{equation}
\bSi^L_{H^\bp}= c_m \lambda^{m+4} \Sigma_{3,1,1} +O(\lambda^{m+3}),\;\;\mbox{uniformly in $\bp$, }
\label{eq: hess-asy}
\end{equation}
where the positive definite, symmetric bilinear form  $\bSi_{3,1,1}:\eS_m\dual\times \eS_m\dual\ra \bR$ is described by the equalities (\ref{eq: cov-diag}) and (\ref{eq: cov-off}).   We  denote by $\Gamma_{3,1,1}$ the centered  Gaussian measure on $\eS_m$ with covariance form $\Sigma_{3,1,1}$.

The equality (\ref{eq: omega}) coupled with (\ref{eq: bin}) imply that the covariance operator $\bOm_\bp^L$  satisfies
\begin{equation}
\bOm_\bp^L=O(\lambda^{m+2}),\;\;\mbox{uniformly in $\bp$. }
\label{eq: bom-asy}
\end{equation}
Using (\ref{eq: met-asy1}), (\ref{eq: hess-asy}) and (\ref{eq: bom-asy})   we  deduce that the  covariance operator $\bXi^L_\bp$ defined as in  (\ref{eq: rand-matrix}) satisfies the estimate
\begin{equation}
\bXi_p^L= c_m \lambda^{m+4} \widehat{Q}_{3,1,1}+ O(\lambda^{m+2}),\;\;\mbox{as}\;\;L\ra \infty,\;\;\mbox{uniformly in $\bp$,}
\label{eq: hess-asy1}
\end{equation}
where $\widehat{Q}_{3,1,1}$ is the covariance operator associated to the covariance form $\Sigma_{3,1,1}$ and it is described explicitly in (\ref{eq: cov-op}).  If we denote by $d\Gamma_L$ the Gaussian measure on $\eS_m$ with covariance operator $\bXi^L_\bp$,  we deduce that
\[
d\Gamma_L(Y)=\frac{1}{(2\pi)^{\frac{N_m}{2}}(\det \bXi_\bp^L)^{\frac{1}{2}}} e^{-\frac{(\bXi_\bp^LY,Y)}{2}}\,\cdot\, \underbrace{2^{\frac{1}{2}\binom{m}2}\prod_{i\leq j} dy_{ij}}_{|dY|},
\]
where
\[
N_m=\dim\eS_m=\frac{m(m+1)}{2}.
\]
Let us observe that $|dY|$ is  the Euclidean volume element on $\eS_m$ defined by  the natural inner product on $\eS_m$, $(X,Y)=\tr(XY)$.  We set
\[
c_L:= c_m\lambda^{m+4},\;\;Q^L_\bp=\frac{1}{c_L} \bXi_\bp^L.
\]
Using (\ref{eq: resc-gauss})  we deduce that
\[
 \frac{1}{(2\pi)^{\frac{N_m}{2}}(\det \bXi_\bp^L)^{\frac{1}{2}}}\int_{\eS_m} |\det Y| e^{-\frac{(\bXi_\bp^LY,Y)}{2}} |dY|= \frac{(c_L)^{\frac{m}{2}}}{(2\pi)^{\frac{N_m}{2}}(\det Q_\bp^L)^{\frac{1}{2}}}  \int_{\eS_m} |\det Y| e^{-\frac{(Q_\bp^LY,Y)}{2}} |dY|.
\]
From the estimate (\ref{eq: hess-asy1}) we deduce that
\[
Q_\bp^L\ra \widehat{Q}_{3,1,1}\;\;\mbox{as}\;\;L\ra \infty,\;\;\mbox{uniformly in $\bp$. }
\]
We conclude that
\begin{equation}
\bsE(|\det Y_\bp|)=\int_{\eS_m} |\det Y| d\Gamma_L(Y) \sim c_m^{\frac{m}{2}}\lambda^{\frac{m(m+4)}{2}}\int_{\eS_m} |\det Y| d\Gamma_{3,1,1}(Y).
\label{eq: det-asy}
\end{equation}
The measure $d\Gamma_{3,1,1}$ is   described explicitly in (\ref{eq: cov-gauss1}), more precisely
\[
d\Gamma_{3,1,1}(Y)= \frac{1}{(2\pi)^{\frac{N_m}{2}}\sqrt{\mu_m}}\,\cdot\,  e^{-\frac{1}{4}\bigl(\,\tr Y^2-\frac{1}{m+2}(\tr Y)^2\,\bigr)}|dY|,
\]
where  $\mu_m$ is given by (\ref{eq: mum}).   Using (\ref{eq: KR3}), (\ref{eq: jac-asy}) and (\ref{eq: det-asy}) we deduce that
\[
\begin{split}
\bsE(\eN_L)&\sim\left(\frac{c_m}{K_m}\right)^{\frac{m}{2}}\lambda^{\frac{m(m+4)}{2}-\frac{m(m+2)}{2}}{\rm vol}_g(M)\int_{\eS_m} |\det Y| d\Gamma_{3,1,1}(Y)\\
&\stackrel{(\ref{eq: weyl})}{ \sim }\left(\frac{c_m}{K_m}\right)^{\frac{m}{2}}\frac{(2\pi)^m}{\bom_m}\dim \bsU_L.
\end{split}
\]
Observe that
\[
\frac{c_m}{K_m}=\frac{\Gamma(2+\frac{m}{2})}{2\Gamma(3+\frac{m}{2})}=\frac{1}{m+4},\;\;\bom_m=\frac{\pi^{\frac{m}{2}}}{\Gamma(1+\frac{m}{2})}\;\;\frac{(2\pi)^m}{\bom_m}= (4\pi)^{\frac{m}{2}}\Gamma\Bigl(1+\frac{m}{2}\Bigr).
\]
This completes the proof of (\ref{eq: main1}) and (\ref{eq: main4}).\qed

\subsection{On the asymptotic behavior of the stochastic metric} \label{ss: asym-met} We denote by $g(L)$ the metric
\[
g(L):= \lambda^{-(m+2)}\si^L=  L^{-\frac{(m+2)}{2}}\bsi^L,
\]
where $K_m$ is described by (\ref{eq: km}).The estimate (\ref{eq: met-asy})  shows that  
\[
g(L)\stackrel{C^0}{\Lra} g\;\;\mbox{as $L\ra \infty$},
\]
where $K_m$ is described by (\ref{eq: km}). The metrics $g(L)$ are closely related to the metrics constructed in \cite[Thm. 5]{BBG}. We want to discuss  here  possible ways to improve the topology of the convergence.   

Observe that if  $g(L)$ were to converge  in the $C^2$-topology  to  $K_m$ then the    sectional curvatures  of $g(L)$ would have  to be uniformly bounded.   Conversely, the results of S. Peters \cite{Pet} show that the $C^0$ convergence coupled  with an uniform bound on the sectional curvatures would   yield a $C^{1,\alpha}$ convergence.

The results  in \cite[\S 12.2.1]{AT} describe   a  simple way of expressing the sectional curvatures of $\si^L$ in terms of the spectral function $\eE_L$.  Here are the details.   

Denote by $\nabla^L$ the Levi-Civita connection of the metric  $\si^L$. Fix a point $\bp\in M$ and $g$-normal coordinates    $(x^1,\dotsc, x^m)$ at $\bp$. We  set
\[
\begin{split}
\eE^L_{i_1,\dotsc, i_a; j_1,\dotsc, j_b}:=\frac{\pa^{a+b}\eE_L(x,y)}{\pa x^{i_1}\cdots \pa x^{i_a}\pa y^{j_1}\cdots \pa y^{j_b}}|_{x=y=0},\\
\\
\bsi(L)_{ij}:= \bsi^L_\bp(\pa_{x^i},\pa_{x^j}),\;\;1\leq i,j\leq m,
\end{split}
\]
and  we denote by $(\, \bsi(L)^{ij}\,)_{1\leq i,j\leq m}$ the  inverse matrix of $( \,\bsi(L)_{ij}\,)_{1\leq i,j\leq m}$.  From \cite[Eq. (12.2.6)]{AT} we deduce
\[
\Gamma(L)_{ijk}:=\bsi^L_{\bp}(\nabla^L_{\pa_{x^i}}\pa_{x^j}, \pa_{x^k}) = \eE^L_{ij;k}.
\]
We set
\[
\Gamma(L)^k_{ij}:=\sum_\ell \bsi(L)^{k\ell}\Gamma(L)_{ij\ell}=\sum_\ell \bsi(L)^{k\ell}\eE^L_{ij;\ell},
\]
so that
\[
\bigl(\, \nabla^L_{\pa_{x^i}}\pa_{x^j}\,\bigr)_\bp=\sum_k \Gamma(L)^k_{ij}\pa_{x^k}.
\]
For  $\bu\in \bsU_L$  we set
\begin{equation}
H^L_{ij}(\bu):=  \bigl(\,\pa_{\pa_x^i}\pa_{x^j}\bu -(\nabla^L_{\pa_{x^i}}\pa_{x^j})\bu\,\bigr)_\bp=\pa_{\pa_x^i}\pa_{x^j}\bu(\bp)-\sum_k \Gamma(L)^k_{ij} \pa_{x_k}\bu(\bp).
\label{eq: hess-sto}
\end{equation}
We think of the matrix $H^L_{ij}(\bu)$  as an element  $H^L(\bu)\in T^*_\bp M\otimes T^*_\bp M$,
\[
H^L(\bu)=\sum_{i,j} H^L_{ij} dx^i\otimes dx^j
\]
 and we set
 \[
 H^L(\bu)\wedge  H^L(\bu):=\sum_{i,j,k\ell} H^L_{ij}(\bu)H^L_{k\ell}(\bu) dx^i\wedge dx^k \otimes dx^j\wedge dx^\ell
 \]
 \[
 =:\sum_{i<k, j<\ell} Q^L_{ikj\ell}(\bu)  dx^i\wedge dx^k \otimes dx^j\wedge dx^\ell.
\]
Note that
\[
Q^L_{ikj\ell}(\bu)=2(\, H_{ij}^L(\bu) H_{k\ell}^L(\bu)- H_{kj}^L(\bu)H_{i\ell}^L(\bu)\,\bigr).
\]
We denote by $R^L$ the Riemann tensor of $\bsi^L$ and we set
\[
R^L_{ijk\ell}:=\bsi^L\bigl(\, R^L(\pa_{x^i},\pa_{x^j})\pa_{x^k},\pa_{x^\ell}\,\bigr)_\bp.
\]
The map $\bsU_L\ni \bu\mapsto Q^L_{ikj\ell}(\bu)\in\bR$ is  a  random variable and according to \cite[Lemma 12.2.1]{AT} we have\footnote{Alternatively, in our case,  the     equalities (\ref{eq: curv-hess}) are   simple consequences of Theorema Egregium, \cite[\S 4.2.4, Eq. (4.2.12)]{N0}.}
\begin{equation}
2R^L_{ikj\ell}=-\bsE\bigl(\, Q^L_{ikj\ell}\,\bigr).
\label{eq: curv-hess}
\end{equation}
In particular we deduce that
\[
-R^L_{ijij} = \bsE\bigl(\, H_{ii}^LH^L_{jj} - (H_{ij}^L)^2\,\bigr).
\]
From (\ref{eq: met-asy}) we deduce that
\[
\bsi(L)_{ij}=\eE^L_{i;j}\sim K_m\lambda^{m+2}\delta_{ij}+O(\lambda^{m+1})\;\;\mbox{as $L\ra\infty$}.
\]
Hence
\[
\bsi(L)^{ij}\sim \frac{1}{K_m\lambda^{m+2}}\Bigl(\delta^{ij}+O(\lambda^{-1})\,\Bigr).
\]
From  (\ref{eq: bin}) we deduce that as $\lambda\ra \infty$ we have
\begin{subequations}
\begin{equation}
\Gamma(L)^k_{ij}\sim  \sum_\ell \frac{1}{K_m\lambda^{m+2}}\Bigl(\delta^{k\ell}+O(\lambda^{-1})\,\Bigr)\eE^L_{ij;\ell}\sim \frac{1}{K_m\lambda^{m+2}}\eE^L_{ij;k}+O(\lambda^{-1})=  O(1),
\label{eq: Gamaa}
\end{equation}
\begin{equation}
\bsE(\pa^2_{x^ix^j}\bu(\bp),\pa_{x^k}\bu(\bp)\,)=\eE^L_{ij;k}= O(\lambda^{m+2})
\label{eq: Gamab}
\end{equation}
\end{subequations}
Using the  estimates  (\ref{eq: cov-hess}), (\ref{eq: hess-sto}), (\ref{eq: Gamaa}) and (\ref{eq: Gamab}) in (\ref{eq: curv-hess})   we  deduce
 \[
\bsE\bigl(\, H_{ii}^LH^L_{jj}- (H^L_{ij})^2\,\bigr)  = \bigl(\eE^L_{ii;jj}-\eE^L_{ij;ij}\,\bigr) + O(\lambda^{m+2}).
\]
We deduce that the sectional curvature of $\bsi^L$ along the plane spanned  by $\pa_{x^i},\pa_{x^k}$ is
\[
K^L_{ij}=-\frac{R_{ijij}}{\bsi(L)_{ii}\bsi(L)_{jj}-\bsi(L)_{ij}^2}=\frac{1}{K_m^2\lambda^{2m+4}}\left(\eE^L_{ii;jj}-\eE^L_{ij;ij}\right) +O\left(\frac{\eE^L_{ii;jj}-\eE^L_{ij;ij}}{\lambda^{2m+5}}\right).
\]
On the other hand
\[
\bsE(\pa^2_{x^ix^j}\bu(\bp),\pa^2_{x^kx^\ell}\bu(\bp)\,)=\eE^L_{ij;k\ell}\sim C_m(i,j;k,\ell)\lambda^{m+4}+O(\lambda^{m+3}),\;\;i\leq j,\;\;k\leq \ell,
\]
where $C_m(i,j;k,\ell)$ is defined by (\ref{eq: ijkl}), and we deduce
\begin{equation}
\eE^L_{ii;jj}-\eE^L_{ij;ij}= \bigl(\,C_m(i,i;j,j)-C_m(i,j;i,j)\,\bigr)\lambda^{m+4}+O(\lambda^{m+3})=O(\lambda^{m+3}).
\label{eq: eijij}
\end{equation}
Hence
\[
K^L_{ij}=\frac{1}{K_m^2\lambda^{2m+4}}\left(\eE^L_{ii;jj}-\eE^L_{ij;ij}\right)+ O(\lambda^{-m-2}).
\]
The sectional curvature of  $g(L)=\lambda^{-m-2}\bsi^L$ along the plane spanned by $\pa_{x^i},\pa_{x^j}$ is
\[
\overline{K}^L_{ij}=\lambda^{m+2}K^L_{ij}= \frac{1}{K_m^2\lambda^{m+2}}\left(\eE^L_{ii;jj}-\eE^L_{ij;ij}\right)+ O(1).
\]
We deduce that the sectional curvatures of $g(L)$ are uniformly bounded if and only if
\begin{equation}
\eE^L_{ii;jj}-\eE^L_{ij;ij}= O(\lambda^{m+2})\;\;\mbox{uniformly over $M$}.
\label{eq: bcurv}
\end{equation}
Note that the estimates (\ref{eq: bcurv}) are stronger than the  estimates (\ref{eq: eijij})   which  are direct consequences of the Bin-H\"{o}rmander estimates (\ref{eq: bin}).

Let us point our that (\ref{eq: bcurv}) hold when $(M,g)$ is a homogeneous space equipped with an invariant metric. Indeed, in this case the metric $g(L)$ has the same symmetries as $g$ and thus there exists  a constant $c_L>0$ such that $g(L)=c_L g$. Then $c_L\ra K_m$   as $L\ra \infty$ so that $g(L)\ra K_mg$ in the $C^\infty$-topology and therefore  $\overline{K}^L_{ij}= O(1)$.

\section{The proof of Theorem \ref{th: main2}}
\label{s: asy}
\setcounter{equation}{0}

\subsection{Reduction to  the classical   Gaussian orthogonal ensemble.}We  begin  by describing  the large $m$ behavior of the integral
\[
I_m:=\frac{1}{(2\pi)^{\frac{m(m+1)}{4}}\sqrt{\mu_m}}\int_{\eS_m}|\det X| e^{-\frac{1}{4}\bigl(\,\tr X^2-\frac{1}{m+2}(\tr X)^2\,\bigr)}|dX|,
\]
where we recall that
\[
\mu_m=2^{\binom{m}{2}+m-1}(m+2).
\]
We will use  a trick of Fyodorov \cite{Fy}; see also \cite[\S 1.5]{For}. Recall first the classical  equality
\[
\int_\bR  e^{-(at^2+bt+c)} |dt|= \left(\frac{\pi}{a}\right)^{\frac{1}{2}}e^{\frac{\Delta}{4a}},\;\;\Delta=b^2-4ac,\;\;a>0.
\]
For any real numbers  $u,v,w$,  we have
\[
\begin{split}
ut^2+v\tr(X+wt\one_m)^2 &=  (u+mw^2)t^2+ 2vw(\tr X)t+ v \tr X^2\\
&=:a(u,v,w)t^2+b(u,v,w)t+c(u,v,w).
\end{split}
\]
We seek $u,v,w$ such that
\[
 \frac{v^2w^2}{u+mw^2}(\tr X)^2-\frac{v}{u+mw^2}\tr X^2=\frac{b^2-4ac}{4a}=-\frac{1}{4}\left(\,\tr X^2-\frac{1}{m+2}(\tr X)^2\right).
\]
We have
\[
\frac{v}{u+mw^2}=\frac{1}{4},\;\; \frac{v^2w^2}{u+mw^2}=\frac{1}{4(m+2)},
\]
and we deduce
\[
vw^2=\frac{1}{(m+2)},\;\;v= \frac{1}{4}(u+mw^2) \Llra u=4v-mw^2.
\]
Hence
\[
w^2= \frac{1}{v(m+2)},\;\; u= 4v-\frac{m}{v(m+2)}.
\]
We choose $v=\frac{1}{2}$ so that
\[
w^2=\frac{2}{(m+2)},\;\;u=2-\frac{2m}{(m+2)}=\frac{4}{m+2}, \;\;a(u,v,w)=4v=2,
\]
\[
e^{-\frac{1}{4}\left(\,\tr X^2-\frac{1}{m+2}(\tr X)^2\right)} =\left(\frac{2}{\pi}\right)^{\frac{1}{2}}\int_{\bR} e^{-\frac{4t^2}{m+2}} e^{-\frac{1}{2}\tr(X+t \sqrt{ \frac{2}{m+2} } \one_m)^2}dt
\]
($s=\sqrt{\frac{2}{m+2}}t$)
\[
=\left(\frac{2(m+2)}{\pi}\right)^{\frac{1}{2}}\int_{\bR} e^{-\frac{1}{2}\tr(X+\frac{s}{\sqrt{2}}\one_m)^2} e^{-s^2}ds=\left(\frac{m+2}{2}\right)^{\frac{1}{2}}\int_{\bR} e^{-\frac{1}{2}\tr(X-s\one_m)^2}\,\cdot\, \underbrace{\frac{e^{-2s^2}}{\sqrt{\frac{\pi}{2}}}ds}_{d\bgamma(s)}.
\]
Hence
\begin{equation}
\begin{split}
I_m &=  \underbrace{\frac{(m+2)^{\frac{1}{2}}}{2^{\frac{1}{2}}(2\pi)^{\frac{m(m+1)}{4}}\sqrt{\mu_m}}}_{=:A_m}\,\int_{\bR}\left(\int_{\eS_m} |\det X| e^{-\frac{1}{2}\tr(X-s\one_m)^2}|dX|\right) d\bgamma(s)\\
&=A_m\int_{\bR}\underbrace{\left(\int_{\eS_m} |\det (x\one_m -Y)| e^{-\frac{1}{2}\tr Y^2}|dY|\right) }_{=:f_m(x)}d\bgamma(x).
\end{split}
\label{eq: fm}
\end{equation}
For any $O(n)$-invariant  function $: \eS_n\ra \bR$ we have a  Weyl integration formula (see \cite{AGZ, For, Me}),
\[
\frac{1}{(2\pi)^{\frac{\dim\eS_m}{2}}}\int_{\eS_n} f(X) |dX| =     \frac{1}{\bsZ_n}\int_{\bR^n} f(\lambda) |\Delta_m(\lambda)|\,|d\lambda|,
\]
where
\[
\Delta_n(\lambda):=\prod_{1\leq i< j\leq n} (\lambda_j-\lambda_i),
\]
and the constant  $\bsZ_n$ is defined by the equality \cite[Eq. (2.5.11)]{AGZ} ,
\begin{equation}
\bsZ_m:=\int_{\bR^n} e^{-\frac{1}{2}|\lambda|^2} |\Delta_m(\lambda)|\,|d\lambda|=2^{\frac{n}{2}} n!\prod_{j=1}^n\Gamma\Bigl(\frac{j}{2}\Bigr).
\label{eq: zn}
\end{equation}
Now observe that for any $\lambda_0\in\bR$ we have (with  $f_m$ defined  in (\ref{eq: fm}))
\[
f_m(\lambda_0)=\frac{(2\pi)^{\frac{\dim\eS_m}{2}}}{\bsZ_m} \int_{\bR^m} e^{-\frac{|\lambda|^2}{2}}\left(\prod_{j=1}^n|\lambda_j-\lambda_0|\right) \bigl|\,\Delta_m(\lambda)\,\bigr|\,|d\lambda|
\]
\[
=\frac{e^{\frac{1}{2}\lambda_0^2}(2\pi)^{\frac{\dim\eS_m}{2}}}{\bsZ_m}\int_{\bR^m} e^{-\frac{1}{2}\sum_{i=1}^m \lambda_i^2} \bigl|\Delta_{m+1}(\lambda_0,\lambda_1,\dotsc,\lambda_m)\,\bigr|\,|d\lambda_1\cdots d\lambda_m|\,
\]
\[
=\frac{e^{\frac{1}{2}\lambda_0^2}(2\pi)^{\frac{\dim\eS_m}{2}}\bsZ_{m+1}}{\bsZ_m}\, \underbrace{\frac{1}{\bsZ_{m+1}}\int_{\bR^m} e^{-\frac{1}{2}\sum_{i=1}^m \lambda_i^2}\bigl|\, \Delta_{m+1}(\lambda_0,\lambda_1,\dotsc,\lambda_m)\,\bigr|\,|d\lambda_1\cdots d\lambda_m|}_{=:\rho_{m+1}(\lambda_0)}.
\]
The function $R_n(x)=n\rho_{n}(x)$ is known in random  matrix theory as the \emph{$1$-point correlation  function} of the Gaussian  orthogonal ensemble of symmetric $n\times n$ matrices, \cite[\S 4.4.1]{DG}, \cite[\S 3]{Fy2}, \cite[\S 4.2]{Me}. We conclude that
\[
I_m =\frac{(2\pi)^{\frac{\dim\eS_m}{2}}A_m\bsZ_{m+1}}{\bsZ_m}\int_{\bR} \rho_{m+1}(x)e^{\frac{x^2}{2}}d\bgamma(x)=\frac{A_m\bsZ_{m+1}}{\bsZ_m}\int_{\bR} \rho_{m+1}(x)\sqrt{\frac{2}{\pi}}e^{-\frac{3x^2}{2}} dx.
\]
We have
\[
\frac{\bsZ_{m+1}}{\bsZ_m}=2^{\frac{1}{2}}(m+1)\Gamma\left(\frac{m+1}{2}\right),
\]
\[
\sqrt{\frac{2}{\pi}}\frac{(2\pi)^{\frac{\dim\eS_m}{2}}A_m\bsZ_{m+1}}{\bsZ_m}= (2\pi)^{\frac{\dim\eS_m}{2}}\sqrt{\frac{2}{\pi}} 2^{\frac{1}{2}}(m+1)\Gamma\left(\frac{m+1}{2}\right) \frac{(m+2)^{\frac{1}{2}}}{2^{\frac{1}{2}}(2\pi)^{\frac{m(m+1)}{4}}\sqrt{\mu_m}}
\]
\[
=\sqrt{\frac{2}{\pi}} (m+1)\Gamma\left(\frac{m+1}{2}\right) \frac{(m+2)^{\frac{1}{2}}}{\sqrt{\mu_m}}
\]
\[
=\sqrt{\frac{2}{\pi}} (m+1)\Gamma\left(\frac{m+1}{2}\right) \frac{(m+2)^{\frac{1}{2}}}{2^{\frac{1}{2}\binom{m-1}{2}+\frac{m-1}{2}}(m+2)^{\frac{1}{2}}}
\]
\[
=\sqrt{\frac{2}{\pi}} (m+1) \frac{\Gamma\left(\frac{m+1}{2}\right)}{2^{\frac{1}{2}\binom{m-1}{2}+\frac{m-1}{2}}}= \sqrt{\frac{2}{\pi}} \frac{2\Gamma\left(\frac{m+3}{2}\right)}{2^{\frac{1}{2}\binom{m-1}{2}+\frac{m-1}{2}}}.
\]
We  deduce
\begin{equation}
I_m =\sqrt{\frac{2}{\pi}} \frac{2\Gamma\left(\frac{m+3}{2}\right)}{2^{\frac{1}{2}\binom{m-1}{2}+\frac{m-1}{2}}}\int_{\bR} \rho_{m+1}(x)   e^{-\frac{3x^2}{2}} dx.
\label{eq: im}
\end{equation}
We set
\[
\bar{\rho}_n(s):=\sqrt{n}\rho_n(\sqrt{n}s),
\]
and we deduce
\begin{equation}
\begin{split}
\int_{\bR}\rho_n(x) e^{-\frac{3x^2}{2}} dx &= \int_{\bR} \rho_n(\sqrt{n} s) e^{-\frac{3ns^2}{2}} ds=n \int_{\bR} e^{-\frac{3ns^2}{2}}\bar{\rho}_n(s) ds\\
&=\left(\frac{2\pi}{3n}\right)^{\frac{1}{2}} \int_{\bR}\underbrace{  \frac{(3n)^{\frac{1}{2}}e^{-\frac{3ns^2}{2}}}{(2\pi)^{\frac{1}{2}}}}_{=:w_n(s)}\, \cdot\, \bar{\rho}_n(s) ds.
\end{split}
\label{eq: rhon}
\end{equation}
To proceed further  we use as guide  Wigner's theorem, \cite{AGZ, DG, For, Me}  stating that the    sequences of probability measures 
\[
\bar{\rho}_n(x)dx=\sqrt{n}\rho_n(\sqrt{n} x) dx=\frac{1}{\sqrt{n}} R_n(\sqrt{n}x)dx
\]
 converges  weakly to the semi-circle  probability measure\footnote{There are different rescalings of the  semicircle measures in the literature. Our conventions agree with those in \cite{Me}.} $\rho(x)dx$,
\begin{equation}
\rho (x)=\frac{1}{\pi}\begin{cases}
\sqrt{2-x^2}, & |x|\leq \sqrt{2}\\
0, & |x|>\sqrt{2}.
\end{cases}
\label{eq:  semi}
\end{equation}
We  observe that the Gaussian measures  $w_n(s) ds$ converge to the   Dirac delta measure concentrated at the origin.  This suggests that
\begin{equation}
\lim_{n\ra\infty}\int_{\bR} \bar{\rho}_n(s) w_n(s) ds =\rho(0)=\frac{\sqrt{2}}{\pi}.
\label{eq: wigner}
\end{equation}
We will show that this is indeed the case by slightly  refining the arguments  in  one particular proof of Wigner's theorem; see  \cite[\S 7.1.6]{For},\cite[\S 6.1]{Fy2} or \cite[A.9]{Me}. For the moment  we  will take (\ref{eq: wigner}) for granted and  show that it immediately   implies  (\ref{eq: main3}).

Using (\ref{eq: wigner})  in (\ref{eq: im})  and (\ref{eq: rhon}) we deduce that
\[
I_m\sim  \sqrt{\frac{2}{\pi}} \frac{2\Gamma\left(\frac{m+3}{2}\right)}{2^{\frac{1}{2}\binom{m-1}{2}+\frac{m-1}{2}}}\times \left(\frac{2\pi}{3(m+1)}\right)^{\frac{1}{2}} \times \frac{\sqrt{2}}{\pi}\;\;\mbox{as}\;\;m\ra \infty.
\]
We now invoke Stirling's formula to conclude that
\begin{equation}
\log I_m \sim  \sim \log \Gamma\left(\frac{m+3}{2}\right)\sim\frac{m}{2}\log m,\;\;\mbox{as}\;\;m\ra \infty.
\label{eq: im2}
\end{equation}
Form (\ref{eq: main2}) we  deduce that
\[
\log C(m)=\log I_m +\frac{m}{2}\log 4\pi +\log \Gamma\Bigl( 1+\frac{m}{2}\Bigr) -\frac{m}{2}\log (m+4).
\]
Stirling's formula and (\ref{eq: im2}) imply  that
\[
\log C(m)\sim \log I_m  \sim\frac{m}{2}\log m\;\;\mbox{as}\;\;m \ra \infty.
\]
This proves (\ref{eq: main3}).\qed

 \subsection{Wigner's semicircle law revisited}   We can now present the postponed proof of (\ref{eq: wigner}). The $1$-point correlation  function  $R_n(x)$ can be expressed explicitly in terms of Hermite polynomials, \cite[Eq. (7.2.32) and \S A.9]{Me},
\begin{equation}
R_n(x)=\underbrace{\sum_{k=0}^{n-1}\psi_k(x)^2}_{=:\kk_n(x)}+\underbrace{\left(\frac{n}{2}\right)^{\frac{1}{2}}\psi_{n-1}(x)\int_{\bR}\ve(x-t)\psi_n(t) dt+ \alpha_n(x)}_{=:\bell_n(x)},
\label{eq: rhon1}
\end{equation}
where
\[
\psi_n(x)=\frac{1}{(2^n n!\sqrt{\pi})^{\frac{1}{2}}} e^{-\frac{x^2}{2}}H_n(x),\;\; H_n(x)= (-1)^ne^{x^2}\frac{d^n}{dx^n} (e^{-x^2}),
\]
\[
\alpha_n(x)=\begin{cases}
0, & n\in 2\bZ,\\
\frac{\psi_{n-1}(x)}{\int_{\bR}\psi_{n-1}(x) dx}, & n\in 2\bZ+1,
\end{cases}
\]
and
\[
\ve(x)=\begin{cases}
\frac{1}{2}, & x>0\\
0, & x=0,\\
-\frac{1}{2}, & x<0.
\end{cases}
\]
From the Christoffel-Darboux formula \cite[Eq. (5.5.9)]{Sze} we deduce
\[
\pi^{\frac{1}{2}} e^{x^2}\sum_{k=0}^{n-1}\psi_k(x)^2=\sum_{k=1}^{n-1}\frac{1}{2^k k!}H_k(x)^2=\frac{1}{2^n(n-1)!} \bigl(\,H_n'(x)H_{n-1}(x)-H_n(x) H_n'(x)\,\bigr)
\]
Using the  recurrence formula $H_n'=2xH_n-H_{n+1}$ we deduce
\[
H_n'(x)H_{n-1}(x)-H_n(x) H_n'(x)=H_n^2(x)- H_{n-1}(x)H_{n+1}(x)
\]
and
\[
\kk_n(x)=\frac{e^{-x^2}}{2^{n} (n-1)!\pi^{\frac{1}{2}}}\Bigl( H_n^2(x)- H_{n-1}(x)H_{n+1}(x)\,\Bigr).
\]
We set
\[
\bar{\kk}_n(x):=\frac{\kk_n\bigl(\sqrt{n} x\bigr)}{\sqrt{n}},\;\;\bar{\bell}_n(x):=\frac{\bell_n\bigl(\sqrt{n} x\bigr)}{\sqrt{n}},\;\;\bar{R}_n(x)=\frac{1}{\sqrt{n}}R_n(\sqrt{n} x)=\bar{\rho}_n(x)
\]
so that
\[
\bar{R}_n(x)=\bar{\kk}_n(x)+\bar{\bell}_n(x).
\]
\begin{lemma}
\begin{equation}
\lim_{n\ra \infty}\sup_{x\in \bR}|\bar{\bell}_n(x)|=0.
\label{eq: lninfty}
\end{equation}
\label{lemma: ln}
\end{lemma}

\begin{proof} Using the generating series  \cite[Eq. (5.5.7)]{Sze}
\[
\sum_{n=0}^\infty H_n(x)\frac{T^n}{n!}=e^{2Tx-T^2}
\]
we deduce that
\[
\sum_{n=0}^\infty\left(\int_{\bR} e^{-\frac{x^2}{2}}H_n(x) dx\right) \frac{T^n}{n!}= e^{T^2} \int_\bR e^{-\frac{(x-2T)^2}{2}} dx=\sqrt{2\pi} e^{T^2},
\]
so that
\[
\frac{1}{(2n)!}\int_\bR e^{-\frac{x^2}{2}} H_{2n}(x) dx=\frac{\sqrt{2\pi}}{n!}\;\;\mbox{and}\;\;\int_\bR\psi_{2n}(x) dx=\frac{\sqrt{2 (2n)!}}{ 2^n n!\pi^{\frac{1}{4}}}\sim  {\rm const}\cdot n^{\frac{1}{4}}\;\;\mbox{as $n\ra \infty$}.
\]
Using \cite[Thm. 6.55]{DG} or  \cite[Thm. 8.91.3]{Sze}  we deduce  that
\[
\sup_{x\in \bR}|\psi_n(x)|= O(n^{-\frac{1}{12}})
\]
and thus 
\[
\sup_{x\in\bR}|\alpha_{n}(x)|= O(n^{-\frac{1}{12}-\frac{1}{4}})=O(n^{-\frac{1}{3}})\;\;\mbox{as $n\ra \infty$}.
\]
We set
\[
F_n(x)=\int_\bR \ve(x-t)\psi_n(t) dt.
\]
Using \cite[Thm. 6.55 + Eq. (6.26)]{DG} we deduce $\sup_{x\in \bR}|F_n(x)|= O\bigl(\,n^{-\frac{1}{12}}\,\bigr)$. This proves (\ref{eq: lninfty}).\end{proof}
From the above lemma we deduce that
\[
\int_{\bR}\bigl(\,\bar{\rho}_n(s)-\rho(s)\,\bigr)w_n(s) ds=\int_{\bR}\bigl(\,\bar{\kk}_n(s)-\rho(s)\,\bigr)w_n(s) ds +O\bigl(\,n^{-\frac{1}{12}}\,\bigr)\;\;\mbox{as $n\ra \infty$}.
\]
\begin{lemma}
\[
\lim_{n\ra \infty}\int_{\bR}\bigl(\,\bar{\kk}_n(s)-\rho(s)\,\bigr)w_n(s) ds=0.
\]
\end{lemma}

\begin{figure}[h]
\centering{\includegraphics[height=1.5in,width=2.5in]{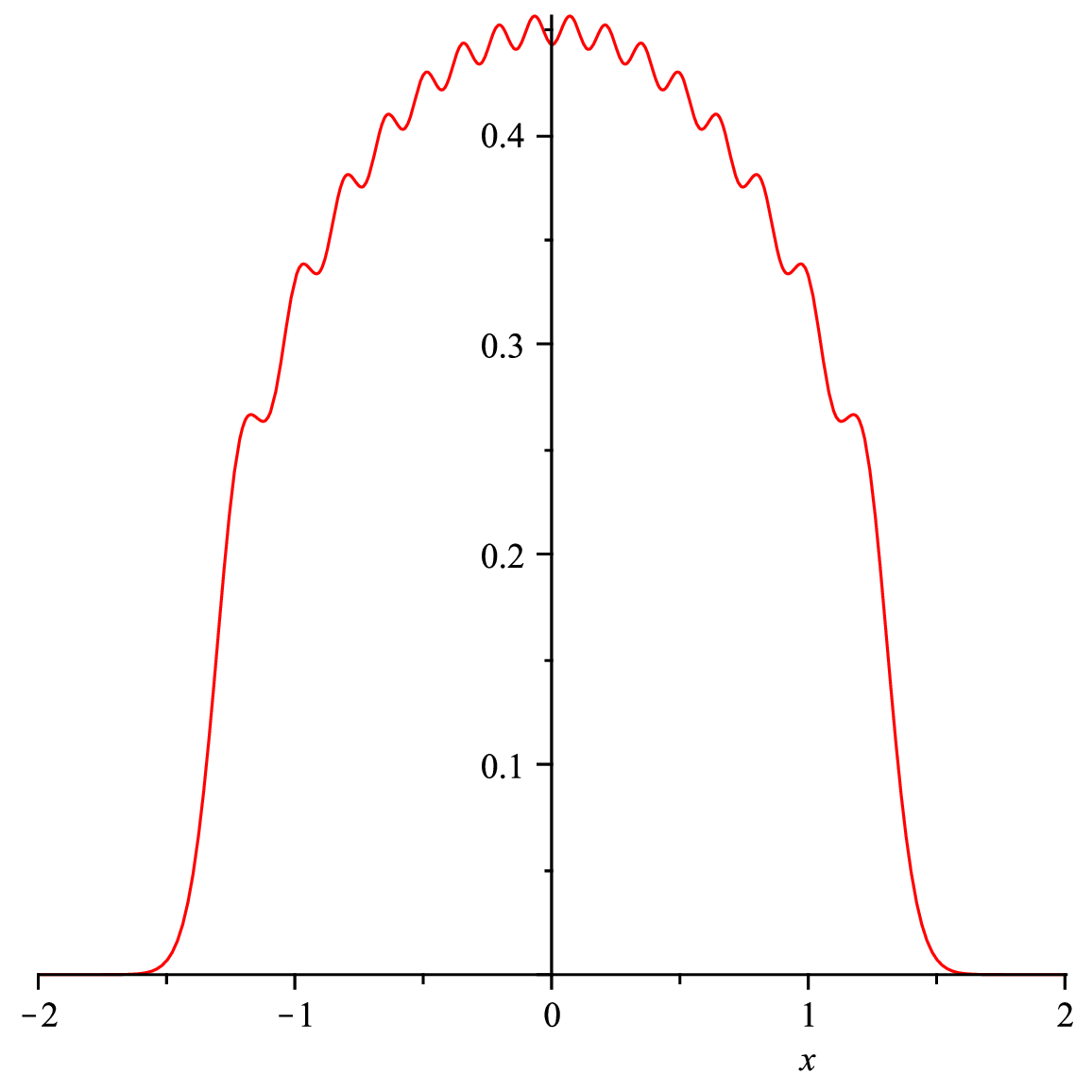}}
\caption{\sl The graph of $\bar{\kk}_{16}(x)$, $|x|\leq 2$.}
\label{fig: hermite}
\end{figure}

\begin{proof}  Fix $c\in (0,\sqrt{2})$ so that  the interval $(-c,c)$ lies inside the oscillatory regime of $H_n(\sqrt{n} t)$. We have
\[
\int_{\bR}\bigl(\,\bar{\kk}_n(s)-\rho(s)\,\bigr)w_n(s) ds
\]
\[
=\int_{|s|\leq c}\bigl(\,\bar{\kk}_n(s)-\rho(s)\,\bigr)w_n(s) ds+\int_{|s|>c}\bigl(\,\bar{\kk}_n(s)-\rho(s)\,\bigr)w_n(s) ds
\]
\[
\leq \sup_{|s|\leq c}|\bar{\kk}_n(s)-\rho(s)| +\sup_{|s|>c}|\bigl(\,\bar{\kk}_n(s)-\rho(s)\,\bigr)|\int_{|s|>c}w_n(s) ds.
\]
Using the  Plancherel-Rotach  formul{\ae} (\cite[Eq. (6.126)]{DG}, \cite{PlRo}, \cite[Thm. 8.22.9]{Sze})  and arguing  as in \cite[\S7.1.6]{For} or  \cite[\S 6.1]{Fy2} we deduce that
\[
\lim_{n\ra \infty}\sup_{|s|\leq c}|\bar{\kk}_n(s)-\rho(s)|=0.
\]
On the other hand
\[
\lim_{n\ra \infty}\int_{|s|>c}w_n(s) ds=0,
\]
 and \cite[Thm.8.91.3]{Sze}  implies that
\[
\sup_{|s|>c}|\bigl(\,\bar{\kk}_n(s)-\rho(s)\,\bigr)|=O(1)\;\;\mbox{as $n\ra \infty$}.
\]
\end{proof}
Since $w_n(s) ds$ converges to the $\delta$-measure concentrated at the origin we deduce
\[
\lim_{n\ra \infty}\int_{\bR}\rho(s) w_n(s) ds= \rho(0)=\frac{\sqrt{2}}{\pi}.
\]
This proves (\ref{eq: wigner}).

\appendix

\section{Gaussian measures and Gaussian random fields}
\label{s: gauss}
\setcounter{equation}{0}

For the reader's convenience we  survey here a few  basic facts about Gaussian  measures. For more details we refer to \cite{Bog}.  A   \emph{Gaussian measure} on $\bR$  is a Borel measure $\gamma_{m,\si}$  of the form
\[
\gamma_{m,\si}(x)= \frac{1}{\si \sqrt{2\pi}} e^{-\frac{ (x-m)^2}{2\si^2}} dx.
\]
The scalar $m$ is called the \emph{mean} while $\si$ is called the \emph{standard deviation}. We allow $\si$ to be zero in which case
\[
\gamma_{m,0}=\delta_m=\mbox{the Dirac measure on $\bR$ concentrated at $m$}.
\]
Suppose that $\bsV$ is a finite dimensional vector space. A \emph{Gaussian measure} on $\bsV$ is a  Borel measure  $\gamma$ on $\bsV$ such that, for any $\xi\in\bsV\dual$, the pushforward $\xi_*(\gamma)$ is a Gaussian measure on $\bR$, $\xi_*(\gamma)=\gamma_{m(\xi),\si(\xi)}$.

The map $\bsV\dual\ni \xi\mapsto m(\xi)\in\bR$ is linear, and  thus can be identified with a  vector $\bm_\gamma\in \bsV$ called the \emph{barycenter} or \emph{expectation} of $\gamma$ that can be alternatively defined by the equality $\bm_\gamma=\int_{\bsV} \bv d\gamma(\bv)$. Moreover, there exists  a   nonnegative definite, symmetric bilinear map
\[
\bSi: \bsV\dual\times\bsV\dual\ra \bR\;\;\mbox{such that}\;\;\si(\xi)^2= \bSi(\xi,\xi),\;\;\forall \xi\in \bsV\dual.
\]
The form  $\bSi$ is called the \emph{covariance form} and can be identified with a  linear operator $\bsS:\bsV\dual\ra \bsV$ such that
\[
\bSi(\xi,\eta)=\lan \xi, \bsS\eta\ran,\;\;\forall \xi,\eta\in \bsV\dual,
\]
where $\lan-,-\ran:\bsV\dual\times\bsV\ra \bR$ denotes the natural bilinear pairing between a vector space and its dual. The operator $\bsS$ is called the \emph{covariance operator} and it is explicitly described by the  integral formula
\[
\lan \xi, \bsS\eta\ran=\Lambda(\xi,\eta)=\int_{\bsV}\lan \xi,\bv-\bm_\gamma\ran \lan\eta, \bv-\bm_\gamma\ran d\gamma(\bv).
\]
The Gaussian measure is said to be \emph{nondegenerate} if $\bSi$ is nondegenerate, and it is called \emph{centered} if $\bm=0$. A nondegenerate    Gaussian measure  on $\bsV$ is uniquely determined by  its covariance form and its  barycenter. 

\begin{ex}  Suppose that   $\bsU$ is an $n$-dimensional  Euclidean space with  inner product $(-,-)$. We use the inner product to identify $\bsU$ with its dual $\bsU\dual$. If $A:\bsU\ra \bsU$ is a symmetric, positive definite   operator,  then
\begin{equation}
d\bgamma_A(\bx) =\frac{1}{(2\pi)^{\frac{n}{2}}\sqrt{\det A}} e^{-\frac{1}{2}(A^{-1}\bu, \bu)}\,|d\bu|
\label{eq: gamaA}
\end{equation}
is a  centered  Gaussian  measure on $\bsU$ with covariance form described by the  operator $A$.\qed
\end{ex}

If  $\bsV$ is a  finite dimensional vector space  equipped with a  Gaussian measure  $\gamma$  and $\bsL: \bsV\ra \bsU$ is a   linear  map then the pushforward $\bsL_*\gamma$ is a  Gaussian measure on   $\bsU$ with barycenter
\[
\bm_{\bsL_*\gamma}=\bsL(\bm_\gamma)
\]
and covariance form
\[
\bSi_{\bsL_*\gamma}:\bsU\dual\times \bsU\dual\ra \bR,\;\; \bSi_{\bsL_*\gamma}(\eta,\eta)= \bSi_\gamma(\bsL\dual\eta,\bsL\dual\eta),\;\;\forall \eta \in \bsU\dual,
\]
where $\bsL\dual:\bsU\dual\ra \bsV\dual$ is the dual (transpose) of the linear map $\bsL$. Observe that if $\gamma$ is nondegenerate and $\bsL$ is surjective, then $\bsL_*\gamma$ is also nondegenerate.

Suppose $(\eS, \mu)$ is a probability space.    A \emph{Gaussian} random vector on $(\eS,\mu)$ is a (Borel) measurable map
\[
X: \eS\ra \bsV,\;\;\mbox{$\bsV$ finite dimensional vector space}
\]
such that $X_*\mu$ is a Gaussian measure on $\bsV$. We will refer to this measure as the \emph{associated Gaussian measure}, we denote it  by $\gamma_X$ and  we denote by $\bSi_X$ (respectively $\bsS_X$) its covariance form (respectively operator),
\[
\bSi_X(\xi_1,\xi_2)=\bsE\bigl(\, \lan \xi_1, X-\bsE(X)\,\ran\,\lan \xi_2, X-\bsE(X)\,\ran\,\bigr).
\]
 Note that the   expectation of $\gamma_X$ is precisely the expectation of $X$. The random vector is called \emph{nondegenerate}, respectively \emph{centered}, if the Gaussian measure  $\gamma_X$ is such.

 Suppose  that $X_j:\eS\ra \bsV_1$, $j=1,2$, are two \emph{centered} Gaussian random vectors such that the direct sum $X_1\oplus X_2:\eS\ra \bsV_1\oplus \bsV_2
$ is also a centered  Gaussian random vector with  associated  Gaussian measure
\[
\gamma_{X_1\oplus X_2}= p_{X_1\oplus X_2} (\bx_1,\bx_2) |d\bx_1d\bx_2|.
\]
 We obtain a bilinear form
\[
\cov(X_1,X_2):\bsV_1\dual\times \bsV_2\dual\ra \bR,\;\; \cov(X_1,X_2)(\xi_1,\xi_2)=\bSi(\xi_1,\xi_2),
\]
called the \emph{covariance form}. The   random vectors $X_1$ and $X_2$ are independent if and only if  they are uncorrelated, i.e.,
\[
\cov(X_1,X_2)=0.
\]
We can form the random vector $\bsE(X_1|X_2)$, the conditional expectation of $X_1$ given $X_2$.  If $X_1$ and $X_2$ are independent then $\bsE(X_1|X_2)=\bsE(X_1)$, while at the other extreme we have $\bsE(X_1|X_1)=X_1$.

To find a formula for $\bsE(X_1|X_2)$ in general we fix Euclidean metrics $(-,-)_{\bsV_j}$ on $\bsV_j$. We can then identify  $\cov(X_1,X_2)$ with a linear operator $\Cov(X_1,X_2):\bsV_2\ra \bsV_1$, via the equality
\[
\begin{split}
\bsE\bigl(\,\lan \xi_1,X_1\ran\lan\xi_2, X_2\ran\,\bigr) &=\cov(X_1,X_2)(\xi_1,\xi_2)\\
&=\bigl\lan\, \xi_1, \Cov(X_1,X_2)\xi_2^\dag\,\bigr\ran,\;\;\forall \xi_1\in\bsV_1\dual,\;\;\xi_2\in\bsV_2\dual,
\end{split}
\]
where $\xi_2^\dag\in\bsV_2$ denotes the vector metric dual to $\xi_2$. The  operator $\Cov(X_1,X_2)$ is called the \emph{covariance operator} of $X_1, X_2$.   For a proof of the next classical result we refer to  \cite[Prop. 1.2]{AzWs}.

\begin{lemma}[Regression formula] If $X_1$ and $X_2$ are as above and, additionally, $X_2$ is nondegenerate, then
\begin{equation}
\bsE(X_1|X_2)= \Cov(X_1,X_2)\bsS_{X_2}^{-1}\bigl(\,X_2-\bsE(X_2)\,\bigr) + \bsE(X_1).
\label{eq: regress1}
\end{equation}\qed
\end{lemma}
The conditional  probability density of $X_1$  given that $X_2=\bx_2$ is the function
\[
p_{(X_1|X_2=\bx_2)}(\bx_1)=  \frac{p_{X_1\oplus X_2}(\bx_1,\bx_2)}{\int_{\bsV_1} p_{X_1\oplus X_2}(\bx_1,\bx_2) |d\bx_1|}.
\]
For a measurable function $f:\bsV_1\ra \bR$ the conditional expectation   $\bsE(f(X_1)|X_2=\bx_2)$ is the (deterministic) scalar
\[
\bsE(f(X_1)|X_2=\bx_2)= \int_{\bsV_1} f(\bx_1) p_{(X_1|X_2=\bx_2)}(\bx_1)|d\bx_1|.
\]
Again, if $X_2$ is nondegenerate,  then we have the \emph{regression formula}
\begin{equation}
\bsE(f(X_1)|X_2=\bx_2) =\bsE\bigl(\, f(Y+C\bx_2)\,\bigr)
\label{eq: regress2}
\end{equation}
where $Y:\eS\ra \bsV_1$ is a Gaussian  vector with
\begin{equation}
\bsE(Y)=\bsE(X_1)-C\bsE(X_2),\;\;\bsS_Y=\bsS_{X_1}-\Cov(X_1,X_2)\bsS_{X_2}^{-1} \Cov(X_2,X_1),
\label{eq: cov-regr}
\end{equation}
and $C$ is given by 
\begin{equation}
 C= \Cov(X_1,X_2)\bsS_{X_2}^{-1}.
\label{eq: gauss-c}
\end{equation}
Let us point out that if $X:\eS\ra \bsU$ is a Gaussian random vector and $\bsL:\bsU\ra \bsV$ is a linear map, then the random vector $\bsL X:\eS\ra \bsV$ is also Gaussian. Moreover
\[
\bsE(\bsL X)= \bsL \bsE(X),\;\;\bSi_{\bsL X}(\xi,\xi)=\bSi_X(\bsL\dual\xi,\bsL\dual\xi),\;\;\forall \xi\in\bsV\dual,
\]
where $\bsL\dual: \bsV\dual\ra \bsU\dual$ is the linear map dual to $\bsL$. Equivalently,  $\bsS_{\bsL X}= \bsL \bsS_X \bsL\dual$.

A \emph{random field} (or \emph{function})  on a set $\bsT$ is a map $\xi:\bsT \times (\eS, \mu) \ra \bR,\;\; (t, s)\mapsto\xi_t(s)$ such that
 \begin{itemize}

 \item $(\eS,\mu)$ is a probability space,   and

 \item for any $t\in\bsT$ the  function $\xi_t:\eS\ra \bR$ is measurable, i.e., it is a random variable.

 \end{itemize}

Thus, a random field on $\bsT$ is a family  of random variables $\xi_t$ parameterized by the set $\bsT$.   For simplicity we will assume that all these random variables have finite second moments.   For any $t\in \bsT$ we denote by $\mu_{t_1}$ the expectation  of $\xi_t$.   The  \emph{covariance function} or \emph{kernel} of the  field   is the function $C_\xi: \bsT\times \bsT\ra \bR$ defined by
\[
C_\xi(t_1,t_2)= \bsE\bigl(\,(\xi_{t_1}-\mu_{t_1})\,(\,\xi_{t_2}-\mu_{t_2})\,\bigr)=\int_{\eS} \bigl(\xi_{t_1}(s)-\mu_{t_1}\bigr)\bigl(\xi_{t_2}(s)-\mu_{t_2}\,\bigr)\,d\mu(s).
\]
The  field is called  \emph{Gaussian} if for any  finite subset $F\subset \bsT$  the  random vector
\[
\eS\in s\mapsto  \bigl(\, \xi_t(s)\,\bigr)_{t\in F}\in \bR^F
\]
is a Gaussian   random vector.    Almost all the important information concerning a Gaussian  random field can be extracted from its  covariance kernel.  For more information about random fields we refer to \cite{AT,AzWs,CrLe, GiSk}.

In the conclusion of this   section we want to  describe a few simple integral formulas.

\begin{proposition}  Suppose $\bsV$ is an Euclidean  space of dimension $N$, $f: \bsU\ra \bR$ is a locally integrable, positively homogeneous function of degree $k\geq 0$, and $A:\bsU\ra \bsU$ is a positive definite symmetric operator.   Denote by $B(\bsU)$ the unit ball of $\bsV$ centered at the origin, and by $S(\bsU)$ its boundary. Then the following hold
\begin{equation}
\begin{split}
\frac{1}{\pi^{\frac{N}{2}}(k+N)}\int_{S(\bsU)} f(\bu) |dA(\bu)|&=\frac{1}{\pi^{\frac{N}{2}}}\int_{B(\bsU)} f(\bu) |d\bu|\\
&=\frac{1}{\Gamma\bigl(1+\frac{k+N}{2})}\int_{\bsU} f(\bu) \frac{e^{-|\bu|^2}}{\pi^{\frac{N}{2}}}|d\bu|.
\end{split}
\label{eq: int-ball}
\end{equation}
\begin{equation}
\int_{\bsU} f(\bu) d\bgamma_{tA}(\bu)= t^{\frac{k}{2}}\int_{\bsU} f(\bu) d\bgamma_{A}(\bu)\;\;\forall t>0,
\label{eq: resc-gauss}
\end{equation}
where $d\bgamma_A$ is the Gaussian measure defined by (\ref{eq: gamaA}).
\label{prop: int-form}
\end{proposition}

\begin{proof} We have
\[
\int_{B(\bsU)} f(\bu) |d\bu|  =\int_0^1t^{k+N-1}\left(\int_{S(\bsU)} f(\bu) |dA(\bu)|\right)= \frac{1}{k+N}\int_{S(\bsU)} f(\bu) |dA(\bu)|.
\]
On the other hand
\[
\frac{1}{\pi^{\frac{N}{2}}}\int_{\bsU} f(u) e^{-|\bu|^2} |d\bu|=\frac{1}{\pi^{\frac{N}{2}}}\left(\int_0^\infty t^{k+N-1}  e^{-t^2} dt\right) \int_{S(\bsU)} f(\bu) |dA(\bu)|
\]
\[
=\frac{1}{2\pi^{\frac{N}{2}}}\Gamma\left(\frac{k+N}{2}\right)\int_{S(\bsU)} f(\bu) |dA(\bu)|=\frac{k+N}{2\pi^{\frac{N}{2}}}\Gamma\left(\frac{k+N}{2}\right)\int_{B(\bsU)} f(\bu) |d\bu|.
\]
\[
=\frac{1}{\pi^{\frac{N}{2}}}\Gamma\left(1+\frac{k+N}{2}\right)\int_{B(\bsU)} f(\bu) |d\bu|.
\]
This proves  (\ref{eq: int-ball}).  The  equality (\ref{eq: resc-gauss}) follows by using the change in variables $\bu =t^{\frac{1}{2}}\bv$.
\end{proof}

\section{Gaussian random symmetric matrices}
\label{s: rand-sym}
\setcounter{equation}{0}

  We want to describe in some detail a $3$-parameter family of centered Gaussian measures on  $\eS_m$, the vector space of real symmetric  $m\times m$ matrices, $m>1$.    
  
  For any  $1\leq i\leq j$ define  $\xi_{ij}\in \eS_m\dual$  so that for any $A\in  \eS_m$
\[
\xi_{ij}(A) =  a_{ij}=\mbox{the $(i,j)$-th entry of the matrix $A$}.
\]
The collection  $(\xi_{ij})_{1\leq i\leq j\leq m}$ is a basis of the dual space  $\eS_m\dual$. We denote by $(E_{ij})_{1\leq i\leq j}$ the dual basis of $\eS_m$. More precisely,   $E_{ij}$ is the symmetric  matrix whose  $(i,j)$ and $(j,i)$ entries are $1$ while all the other entries  are equal to zero.   For any $A\in \eS_m$ we have
\[
A=\sum_{i\leq j} \xi_{ij}(A) E_{ij}.
\]
The space $\eS_m$ is equipped with an inner product
\[
(-,-):\eS_m\times\eS_m\ra \bR,\;\;(A,B)= \tr(AB),\;\;\forall A,B\in\eS_m.
\]
This inner product is invariant with respect to the action of $\SO(m)$ on $\eS_m$. We set
\[
\widehat{E}_{ij}:=\begin{cases}
E_{ij}, & i=j\\
\frac{1}{\sqrt{2}}E_{ij}, & i<j.
\end{cases}.
\]
The collection  $(\widehat{E}_{ij})_{i\leq j}$ is a basis  of $\eS_m$ orthonormal with respect to the above inner product.  We set
\[
\hat{\xi}_{ij}:= \begin{cases}
\xi_{ij}, & i=j\\
\sqrt{2}\xi_{ij}, & i<j.
\end{cases}
\]
The collection $(\hat{\xi}_{ij})_{i\leq j}$  the orthonormal basis of $\eS_m\dual$ dual to $(\widehat{E}_{ij})$.  The volume density induced by this metric is
\[
|dX|:=\prod_{i\leq j} d\widehat{\xi}_{ij}= 2^{\frac{1}{2}\binom{m}{2}}\prod_{i\leq j} dx_{ij}.
\]

To any  numbers  $a,b,c$  satisfying the inequalities
\begin{equation}
a-b,\;\; c,\;\; a+(m-1) b >0.
\label{eq: poz}
\end{equation}
 we will associate a centered Gaussian measure $\Gamma_{a,b,c}$ on $\eS_m$  uniquely determined by its covariance form 
 \[
\bSi=\bSi_{a,b,c}:\eS_m\dual\times \eS_m\dual\ra \bR
\]
defined  as follows:
 \begin{subequations}
\begin{equation}
\bSi(\xi_{ii},\xi_{ii}) = a,\;\; \bSi(\xi_{ii},\xi_{jj})=b,\;\;\forall i\neq j,
\label{eq: cov-diag}
\end{equation}
\begin{equation}
\bSi(\xi_{ij},\xi_{ij})=c,\;\;\bSi(\xi_{ij},\xi_{k\ell})=0,\;\;\forall i<j,\;\; k\leq \ell,\;\;(i,j)\neq (k,\ell).
\label{eq: cov-off}
\end{equation}
\end{subequations}
To see that $\bSi_{a,b,c}$ is positive definite if $a,b,c$ satisfy (\ref{eq: poz})  we decompose $\eS_m\dual$ as a direct sum of subspaces
\[
\eS_m\dual=\eD_m\oplus \eO_m,
\]
\[
\eD_m={\rm span}\,\{\xi_{ii};\;\;1\leq i\leq m\},\;\;\eO_m={\rm span}\,\{\xi_{ij};\;\;1\leq i <j \leq m\},\;\;\dim\eO_m=\binom{m}{2}
\]
With respect to this  decomposition, and the corresponding bases of these subspaces the matrix $Q_{a,b,c}$ describing $\bSi_{a,b,c}$ with respect to the basis $(\xi_{ij})$ has a direct sum decomposition
\[
Q_{a,b,c}= G_m(a,b)\oplus c\one_{\binom{m}{2}},
\]
where $G_m(a,b)$ is the $m\times m$ symmetric matrix whose diagonal entries are equal to  $a$ while all the off diagonal     entries are all equal to $b$.

The the spectrum of $G_m(a,b)$ consists of two eigenvalues: $(a-b)$  with multiplicity  $(m-1)$ and the simple eigenvalue $a-b +mb$.  Indeed,  if  $C_m$ denotes the $m\times m$ matrix  with all entries  equal to $1$, then $G_m(a,b)= (a-b)\one_m+bC_m$. The matrix $C_m$ has rank $1$ and a single nonzero eigenvalue equal to $m$ with multiplicity $1$. This proves that $Q_{a,b,c}$ is positive definite since its spectrum is positive.  We denote by $d\Gamma_{a,b,c}$ the  centered  Gaussian measure   on $\eS_m$ with covariance form $\bSi_{a,b,c}$.

Since $\eS_m$ is equipped with an inner product   we can identify   $\bSi_{a,b,c}$ with a symmetric, positive definite bilinear form  on  $\eS_m$. We would like to     compute the matrix $\widehat{Q}=\widehat{Q}_{a,b,c}$ that describes $\bSi_{a,b,c}$ with respect to the orthonormal  basis  $(\widehat{E}_{ij})_{1\leq i\leq j}$.   We have
\[
\widehat{Q}(\widehat{E}_{ii}, \widehat{E}_{ii})= Q(\hat{\xi}_{ii},\hat{\xi}_{ii})=a,\;\;\widehat{Q}(\widehat{E}_{ii}, \widehat{E}_{jj})= b,\;\;\forall i\neq j,
\]
\[
\widehat{Q}(\widehat{E}_{ij}, \widehat{E}_{ij})=Q(\hat{\xi}_{ij},\hat{\xi}_{ij})= 2Q({\xi}_{ij},{\xi}_{ij})=2c,\;\;\forall i<j,
\]
Thus
\begin{equation}
\widehat{Q}_{a,b,c}=G_m(a,b)\oplus 2c \one_{\binom{m}{2}}.
\label{eq: cov-op}
\end{equation}
If $|-|_{a,b,c}$ denotes the  Euclidean norm on $\eS_m$ determined by $\bSi_{a,b,c}$ then for
\[
A=\sum_{i\leq j} a_{ij}E_{ij}= \sum_i a_{ii}\widehat{E}_{ii}+\sqrt{2}\sum_{i<j} a_{ij}\widehat{E}_{ij}.
\]
we have
\[
|A|_{a,b,c}^2 = a\sum_i a_{ii}^2 +2b\sum_{i<j} a_{ii}a_{jj}+4c\sum_{i<j}a_{ij}^2
\]
\[
= (a-b-2c)\sum_ia_{ii}^2+ b\left(\sum_i a_{ii}\right)^2 + 2c\left(\sum_i a_{ii}^2 +2\sum_{i<j} a_{ij}^2\right)
\]
\[
=(a-b-2c)\sum_ia_{ii}^2+b (\tr A)^2 + 2c\tr A^2.
\]
Observe that when
\begin{equation}
a-b=2c
\label{eq: inv}
\end{equation}
  we have
\begin{equation}
|A|_{a,b,c}^2 =b (\tr A)^2 + 2c\tr A^2
\label{eq: inv2}
\end{equation}
so that the norm $|-|_{a,b,c}$ and the  Gaussian measure $d\Gamma_{a,b,c}$ are $O(m)$-invariant.    Let us point out that the space $\eS_m$ equipped with the  Gaussian measure  $d\Gamma_{2,0,1}$ is  the well known $GOE$, the \emph{Gaussian orthogonal ensemble}.

To obtain a   more concrete description of  $\Gamma_{a,b,c}$   we first   identify  $\bSi_{a,b,c}$ with a   symmetric operator $\widehat{Q}_{a,b,c}: \eS_m\ra \eS_m$. Using (\ref{eq: cov-op}) we deduce that
\[
\widehat{Q}_{a,b,c}= G(a,b)\oplus 2c \one_{\binom{m}{2}}.
\]
Observe  that
\begin{equation}
\det\widehat{Q}_{a,b,c}= (a-b)\bigl(\,a+(m-1)b\,\bigr)^{m-1}(2c)^{\binom{m}{2}},
\label{eq: det-cov}
\end{equation}
and
\begin{equation}
\widehat{Q}_{a,b,c}^{-1}=\widehat{Q}_{a',b',c'}= G_m(a',b')\oplus 2c'\one_{\binom{m}{2}},
\label{eq: inverse-cov}
\end{equation}
where $2c'=\frac{1}{2c}$ and  the real numbers $a',b'$ are determined from  the linear system
\begin{equation}
\left\{
\begin{array}{rcl}
a'-b' & = & \frac{1}{a-b}\\
&&\\
a'+(m-1)b' & = & \frac{1}{a+(m-1)b}.
\end{array}
\right.
\label{eq: inv-cov}
\end{equation}
We then have
\begin{equation}
d\Gamma_{a,b,c}(X)=\frac{1}{(2\pi)^{\frac{m(m+1)}{4}}(\det \widehat{Q}_{a,b,c})^{\frac{1}{2}}} e^{-\frac{1}{2}(\widehat{Q}^{-1}_{a,b,c}X, X)} 2^{\frac{1}{2}\binom{m}2}\prod_{i\leq j} dx_{ij},
\label{eq: cov-gauss}
\end{equation}
where
\begin{equation}
(\widehat{Q}^{-1}_{a,b,c}X, X)= \Bigl(\, a'-b'-\frac{1}{2c}\Bigr)\sum_i x_{ii}^2+ b'(\tr X)^2 +\frac{1}{2c}\tr X^2.
\label{eq: quad-gauss}
\end{equation}
The   special case $b=c>0$, $a=3c$ is particularly important for our considerations. We denote by $(-,-)c$   and respectively $d\Gamma_c$ the inner product  and  respectively the Gaussian measure  on $\eS_m$ corresponding to the covariance form $\Sigma_{3c,c,c}$. 

If we set $\widehat{Q}_c:=\widehat{Q}_{3c,c,c}$ then we deduce from (\ref{eq: inverse-cov}) that
\[
\widehat{Q}_c^{-1}= \widehat{Q}_{a',b',c'}=G_m(a',b')\oplus \frac{1}{2c}\one_{\binom{m}{2}},
\]
where
\[
\left\{
\begin{array}{rcl}
a'-b' & = & \frac{1}{2c}=2c'\\
&&\\
a'+(m-1)b' & = & \frac{1}{(m+2)c}.
\end{array}
\right.
\]
We deduce
\[
mb'= \frac{1}{(m+2)c}-\frac{1}{2c}=-\frac{m}{2c(m+2)}\Rightarrow b'  =-\frac{1}{2c(m+2)}.
\]
Note that the invariance condition (\ref{eq: inv}) $a'-b'=2c'$ is  automatically satisfied  so that
\[
\bigl(\widehat{Q}_c^{-1} X, X\bigr)= \frac{1}{2c}\tr X^2 -\frac{1}{2c(m+2)}(\tr X)^2.
\]
Using (\ref{eq: det-cov}) and (\ref{eq: cov-gauss})  we deduce
\begin{equation}
d\Gamma_c(X)=\frac{1}{(2\pi c)^{\frac{m(m+1)}{4}}\sqrt{\mu_m}}\,\cdot\,  e^{-\frac{1}{4c}\bigl(\,\tr X^2-\frac{1}{m+2}(\tr X)^2\,\bigr)} \underbrace{2^{\frac{1}{2}\binom{m}2}\prod_{i\leq j} |dx_{ij}|}_{|dX|},
\label{eq: cov-gauss1}
\end{equation}
where
\begin{equation}
\mu_m:= 2^{\binom{m}{2}+(m-1)}(m+2).
\label{eq: mum}
\end{equation}     
The inner product $(-,-)_c$   has  the  alternate description
\begin{equation}
\begin{split}
(A,B)_c=  I_c(A, B)&:=4c\int_{\bR^m} (A\bx,\bx)(B\bx,\bx) \frac{e^{-|\bx|^2}}{\pi^{\frac{m}{2}}} |d\bx|\\
&=c\int_{\bR^m} (A\bx,\bx)(B\bx,\bx) \frac{e^{-\frac{|\bx|^2}{2}}}{(2\pi)^{\frac{m}{2}}} |d\bx|,\;\;\forall A,B\in\eS_m.
\end{split}
\label{eq: inv1}
\end{equation}

\section{A Gaussian integral}
\label{s: b}
\setcounter{equation}{0}

\noindent {\bf The proof of (\ref{eq: harm60}).}  We want to find the value of the integral
\[
\bsI= \frac{1}{4(2\pi)^{\frac{3}{2}}}\int_{\eS_2} |\det X|e^{-\frac{1}{4}\bigl(\,\tr X^2 -\frac{1}{4}(\tr X)^2\,\bigr)}\,\cdot\, \sqrt{2}\prod_{1\leq i\leq j\leq 2} dx_{ij}.
\]
We first make the change in coordinates 
\[
x_{11}= x+y, \;\;x_{22}=x-y,\;\;x_{12}=z.
\]
Then
\[
\det X= x^2-y^2-z^2,\;\;\tr X= 2x,\;\;\tr X^2=  2(x^2+y^2+z^2).
\]
Hence
\[
\bsI=\frac{1}{2(2\pi)^{\frac{3}{2}}} \int_{\bR^3} |x^2-y^2-z^2| e^{-\frac{1}{4}(x^2+2y^2+2z^2)} \sqrt{2}|dxdydz|
\]
($x=\sqrt{2}u$)
\[
=\frac{1}{(2\pi)^{\frac{3}{2}}} \int_{\bR^3} |2u^2-y^2-z^2| e^{-\frac{1}{2}(u^2+y^2+z^2)} |dudydz|
\]
\[
=\frac{2}{\pi^{\frac{3}{2}}}\int_{\bR^3} |2u^2-y^2-z^2| e^{-(u^2+y^2+z^2)} |dudydz|.
\]
We now make the change in variables  $y=r\cos\theta$, $y=r\sin\theta$, $r>0$ $\theta\in [0,2\pi)$ and we deduce
\[
\begin{split}
\bsI & =\frac{2}{\pi^{\frac{3}{2}}}\int_{\bR}\int_0^\infty \left(\int_0^{2\pi}|2u^2 -r^2|e^{-u^2+r^2} d\theta\right) rdr du\\
&=\frac{8\pi}{\pi^{\frac{3}{2}}}\int_0^\infty\int_0^\infty|2u^2-r^2| e^{-(u^2+r^2)}r  dr du.
\end{split}
\]
We now make the change in variables
\[
u=t\sin\vfi,\;\;r=t\cos\vfi ,\;\;t>0,\;\;0\leq \vfi\leq \frac{\pi}{2},
\]
and we conclude
\[
\bsI=\frac{8}{\pi^{\frac{1}{2}}}\left(\int_0^\infty   e^{-t^2} t^4 dt\right)\left(\int_0^{\frac{\pi}{2}}| 3\sin^2 \vfi -1|  \cos\vfi d\vfi\right)
\]
($t=\sqrt{s}$, $x=\sin\vfi$)
\[
=\frac{4}{\pi^{\frac{1}{2}}}\left(\int_0^\infty e^{-s} s^{\frac{3}{2}} ds\right) \left(\int_0^1 |3x^2-1| dx\right)=\frac{4}{\pi^{\frac{1}{2}}}\times \Gamma\Bigl(\frac{5}{2}\Bigr)\times \frac{4}{3\sqrt{3}}=\frac{4}{\sqrt{3}}.\proofend
\]


\begin{thebibliography}{XXXXX}

\bibitem{AT} R. Adler, R.J.E. Taylor: {\sl  Random Fields and Geometry},  Springer Monographs in Mathematics, Springer Verlag,  2007.

\bibitem{APF} J.C. \'{A}lvarez Paiva, E. Fernandes: {\sl Gelfand transforms  and Crofton formulas}, Sel. Math., New Ser., {\bf 13}(2007), 369-390.

\bibitem{AGZ} G. W. Anderson, A. Guionnet, O. Zeitouni: {\sl An Introduction to Random Matrices}, Cambridge University Press,  2010.


\bibitem{AzWs} J.-M. Aza\"{i}s, M. Wschebor: {\sl  Level Sets and Extrema of Random Processes}, John Wiley \& Sons, 2009.

\bibitem{BBG} P. B\'{e}rard, G. Besson, S. Gallot: {\sl Embedding Riemannian manifolds by their  heat kernel},  Geom. and Funct. Anal., {\bf 4}(1994), 373-398.

\bibitem{BeMe} P. B\'{e}rard, D. Meyer: {\sl In\'{e}galit\'{e}s isop\'{e}rimetriques et applications}, Ann. Sci. \'{E}cole Norm. Sup. , {\bf 15}(1982), 513-541.




\bibitem{Bin} X. Bin: {\sl  Derivatives of the spectral function and Sobolev norms of eigenfunctions on a closed Riemannian manifold}, Ann. Global. Analysis an Geometry, {\bf 26}(2004), 231-252.




\bibitem{BSZ2}  P. Bleher, B. Shiffman, S. Zelditch: \href{http://front.math.ucdavis.edu/0002.4739}{\sl Universality and scaling of zeros on symplectic manifolds} in \emph{Random Matrix Models and Their Applications}, ed. P. Bleher and A.R. Its, MSRI Publications 40, Cambridge University Press, 2001.

\bibitem{Bog} V. I. Bogachev: {\sl  Gaussian Measures},  Mathematical Surveys and Monographs, vol. 62, American Mathematical Society, 1998.

\bibitem{BZ}  Yu.D. Burago, V.A. Zalgaller: {\sl Geometric Inequalities}, Springer Verlag, 1988.

\bibitem{CL} S.S. Chern, R. Lashof: \href{http://www.jstor.org/stable/2372684}{\sl On the total curvature of immersed manifolds}, Amer. J. Math., {\bf 79}(1957), 306-318.



\bibitem{CrLe}   H. Cram\'{e}r, M.R. Leadbetter: {\sl Stationary and Related Stochastic Processes: Sample Function Properties and Their Applications}, Dover, 2004.

\bibitem{DG} P. Deift,  D. Gioev: {\sl Random Matrix Theory: Invariant Ensembles and Universality},  Courant Lecture Notes, vol. 18, Amer. Math. Soc.,  2009.


\bibitem{DSZ1} M. Douglas, B. Shiffman, S. Zelditch: \href{http://arxiv.org/abs/math/0402326}{\sl  Critical points and supersymmetric vacua},  Comm. Math. Phys., {\bf 252}(2004), 325-358.


\bibitem{DSZ2} M. Douglas, B. Shiffman, S. Zelditch:: \href{http://arxiv.org/abs/math/0406089}{\sl  Critical points and supersymmetric vacua, II: Asymptotics and extremal metrics},  J. Diff. Geom., {\bf 72}(2006), 381-427.







\bibitem{For}  P. J. Forrester: {\sl Log-Gases and Random Matrices},  London Math. Soc. Monographs, Princeton University Press, 2010.

\bibitem{Fy} Y. V. Fyodorov: {\sl Complexity of random energy landscapes, glass transition, and absolute value of the spectral determinant of random matrices},  Phys. Rev. Lett,  {\bf 92}(2004), 240601; Erratum: {\bf 93}(2004), 149901.


\bibitem{Fy2} Y. V. Fyodorov: {\sl  Introduction to random matrix theory:  Gaussian unitary ensemble and beyond}, in the volume {\sl Recent perspectives in random matrix theory and number theory}, 31Ð78, 
London Math. Soc. Lecture Note Ser., 322, Cambridge Univ. Press, Cambridge, 2005. 

\bibitem{GGS}  I.M. Gelfand, M. Graev, Z.Ya. Schapira: {\sl  Differential forms and integral geometry}, Funct. Anal. Appl. {\bf 3}(1969), 101-114.


\bibitem{GiSk} I.I. Gikhman, A.V. Skorohod: {\sl Introduction to the Theory of Random Processes}, Dover Publications, 1996.


\bibitem{Hspec} L. H\"{o}rmander: {\sl On the spectral function of an elliptic operator}, Acta Math. {\bf 121}(1968), 193-218.

\bibitem{H1} L. H\"{o}rmander: {\sl The Analysis of Linear Partial Differential Operators I},  Springer Verlag, 1990.

\bibitem{H3} L. H\"{o}rmander: {\sl The Analysis of Linear Partial Differential Operators III},  Springer Verlag, 1994.

\bibitem{Kac} M. Kac: {\sl  The average number of real roots of a  random algebraic equation}, Bull. A.M.S.  {\bf 49}(1943), 314-320.

\bibitem{LPS} H. Lapointe, I. Polterovich, Yu. Safarov: {\sl Average growth of the spectral function on a Riemannian manifold},  Comm. Part. Diff. Eqs., {\bf 34}(2009), 581-615.


\bibitem{Me} M. L. Mehta: {\sl  Random Matrices}, 3rd Edition,  Elsevier, 2004.

\bibitem{Mil} J.W.  Milnor: \href{http://www.jstor.org/stable/1969467}{\sl  On the total curvature of knots},  Ann. Math., {\bf 52}(1950), 248-257.


\bibitem{Mu}  C. M\"{u}ller: {\sl  Analysis of Spherical Symmetries in Euclidean Spaces},   Appl. Math. Sci. vol. 129, Springer Verlag, 1998.

\bibitem{Nast} M. M. N\u{a}st\u{a}sescu: {\sl The number of ovals of a real plane curve},  Princeton Senior Thesis, 2011

\bibitem{NS} S. Nazarov, M. Sodin: \href{http://front.math.ucdavis.edu/0706.2409}{\sl On the number of  nodal domains of random spherical harmonics}, Amer. J. Math, {\bf 131}(2009), 1337-1357. {\sf arXiv: 0706.2409}.

\bibitem{N0} L.I. Nicolaescu: {\sl Lectures on the Geometry of Manifolds}, 2nd Edition, World Scientific,  2007.

\bibitem{N1}L.I. Nicolaescu: {\sl An Invitation to Morse Theory}, Springer Verlag, 2nd Edition,2011.

\bibitem{N3} L.I. Nicolaescu: \href{http://front.math.ucdavis.edu/1103.1276}{\sl The blowup along the diagonal of the spectral function of the Laplacian}, \textsf{arXiv: 1103.1276}


\bibitem{N2} L.I. Nicolaescu: {\sl  Fluctuations of the number of critical points of random trigonometric polynomials}, Anal. St. Univ. "Al.I. Cuza" Iasi,  {\bf 59}(2013), f.1, p 1-24. DOI: 10.2478/v10157-012-0019-6

\bibitem{Ntorus} L.I. Nicolaescu: {\sl Critical points of multidimensional random Fourier series: variance estimates} \textsf{arXiv: 1310.5571}




\bibitem{Pee} J. Peetre: {\sl A generalization of Courant's nodal domain theorem},  Math. Scand., {\bf 5}(1957), 15-20.


\bibitem{Pet} S. Peters: {\sl Convergence of Riemannian manifolds}, Compositio Math., {\bf 62}(1987), 3-16.

\bibitem{PlRo}  M. Plancherel, W. Rotach: {\sl Sur les valeurs asymptotiques des polynomes d'Hermite $H_n(x)=(-1)^ne^{\frac{x^2}{2}}\frac{d^n}{dx^n}\Bigl(e^{-\frac{x^2}{2}}\Bigr)$}, Comment. Math. Helv., {\bf 1}(1929), 227-254.

\bibitem{Plei}  A. Pleijel: {\sl Remarks on Courant's nodal line theorem},  Comm. Pure Appl. Math. {\bf 9}(1956), 543-550.

\bibitem{Rice} S.O. Rice: {\sl Mathematical analysis of random noise}, Bell System Tech. J. 23 (1944), 282Ð332, and 24 (1945), 46Ð156; reprinted in: Selected papers on noise and stochastic processes, Dover, New  York (1954), pp. 133Ð294.

\bibitem{SV}  Yu. Safarov, D. Vassiliev: {\sl The Asymptotic Distribution of Eigenvalues of Partial Differential Operators}, Translations of Math. Monographs, vol. 155, Amer. Math. Soc., 1997.

\bibitem{SZ0} B. Shiffman, S. Zelditch: {\sl Number variance of random zeros},  Geom. and Funct. Anal. {\bf 18}(2008), 1422-1475.


\bibitem{Sze} G. Szeg\"{o}: {\sl  Orthogonal Polynomials}, Colloquium Publ., vol 23,  Amer. Math. Soc.,  2003.

\bibitem{Z0} S. Zelditch: {\sl Szeg\"{o} kernels and a theorem  of Tian}, Int. Math. Res. Notices, {\bf 6}(1998), 317-331.

\bibitem{Zel} S. Zelditch: {\sl   Real and complex zeros of Riemannian  random waves}, Spectral analysis in geometry and number theory, 321Ð342, Contemp. Math., 484, Amer. Math. Soc., Providence, RI, 2009. \textsf{arXiv:0803.433v1}



\end{thebibliography}
\end{document}